\documentclass[11pt,a4paper]{article}
\usepackage{pdfsync} 

\usepackage{amsfonts,amsmath,amsthm}
\newtheorem{theorem}{Theorem}
\numberwithin{theorem}{section}
\newtheorem{lemma}[theorem]{Lemma}

\theoremstyle{definition}

\newtheorem{remark}[theorem]{Remark}
\numberwithin{equation}{section}
\usepackage{amssymb}
\usepackage{hyperref}
\usepackage{graphicx}
\usepackage{wrapfig}
\usepackage{subcaption}
\usepackage{epsfig}
\usepackage{float,mathrsfs}
\usepackage{enumerate}
\usepackage{enumitem}
\usepackage{bbm}
\usepackage[multiple]{footmisc}
\usepackage{todonotes}

\newtheorem{algorithm}{Algorithm}
\usepackage{algpseudocode}
\usepackage[font=scriptsize]{caption}
\counterwithout{figure}{section}
\usepackage{color}
\definecolor{refkey}{rgb}{0.9451,0.2706,0.4941}\definecolor{labelkey}{rgb}{0.9451,0.2706,0.4941}

\definecolor{darkred}{RGB}{139,0,0}
\definecolor{darkgreen}{RGB}{0,100,0}
\definecolor{darkmagenta}{RGB}{139,0,139}
\definecolor{gray}{RGB}{180,180,180}

\newcommand{\bbN}{{\mathbb{N}}}
\newcommand{\bbR}{{\mathbb{R}}}

\newcommand{\bsnu}{{\boldsymbol{\nu}}}

\newcommand{\bssigma}{{\boldsymbol{\sigma}}}

\newcommand{\bsb}{{\boldsymbol{b}}}
\newcommand{\bsm}{{\boldsymbol{m}}}

\newcommand{\bsv}{{\boldsymbol{v}}}

\newcommand{\calC}{\mathcal{C}}

\newcommand{\calL}{\mathcal{L}}

\newcommand{\mask}[1]{{}}

\title{Approximation of risk-averse optimal feedback control}

\author{Philipp A.~Guth\footnotemark[2]~~and Karl Kunisch\footnotemark[2] $^{,}$\footnotemark[3]}

\newcommand{\oldalglinenumber}{}
\NewDocumentCommand{\NoLNState}{}{%
  \RenewCommandCopy{\oldalglinenumber}{\alglinenumber}
  \RenewDocumentCommand{\alglinenumber}{ m }{}
  \State
  \addtocounter{ALG@line}{-1}
  \RenewCommandCopy{\alglinenumber}{\oldalglinenumber}
}

\begin{document}\maketitle

\begin{abstract}
The challenge of constructing feedback control laws for risk-averse optimal control of partial differential equations (PDEs) with random coefficients is addressed. The control objective composes a tracking-type cost with the nonlinear entropic risk measure. A sequential quadratic programming scheme is derived that iteratively solves linear quadratic subproblems obtained through second-order Taylor expansions of the objective functional, with each subproblem re-centered at the previous iterate. It is shown that this method converges locally quadratically to the unique risk-averse optimal control. This work provides the first rigorous feedback synthesis for risk-averse objectives subject to PDEs with random coefficients.
\end{abstract}



\footnotetext[2]{Johann Radon Institute for Computational and Applied Mathematics, Austrian Academy of Sciences, Altenbergerstra{\ss}e 69, AT-4040 Linz, Austria.}
\footnotetext[3]{Institute of Mathematics and Scientific Computing, Karl-Franzens University Graz, Heinrichstrasse~36, AT-8010 Graz, Austria.}
\footnotetext{{\sc Emails}:
{\small\tt   philipp.guth@ricam.oeaw.ac.at,\quad karl.kunisch@uni-graz.at}}

\section{Introduction}
Optimization and optimal control of systems with uncertain parameters is a challenging task since already small parameter changes may change the characteristics of the system, such as its stability properties. If neglected, the uncertainty may lead to failure of the control objective and hence a careful treatment of the uncertainty is indispensable. 

In uncertainty quantification, partial differential equations (PDEs) with random coefficients have become a popular way of incorporating uncertainty into models for engineering or physical processes. One reason behind their success is the ability to integrate knowledge of governing physical equations while accommodating for randomness, which may reflect, for instance, missing data, measurement uncertainties, material imperfections, or external influences. 

If PDEs with random coefficients appear as constraints of an optimal control problem, the objective function becomes a random variable. In order to be able to perform optimization, risk measures are used, which map the random variable objective function to the real line.
A widely used risk measure is the expected value,~\cite{azmi24, ChenGhattas, Gahururu2023, geiersbach, GKKSS21, kouriheinken, Milz2023, Romisch2024}. The expected value is linear, Fr\'echet differentiable and it belongs to the class of coherent risk measures (\cite{coherent}). Despite its widespread adoption, decisions based solely on the expected value (called risk-neutral decisions) are inadequate in many practical scenarios, as they fail to account for the inherent variability in uncertain outcomes. In many situations it is important to minimize the likelihood of extremely large cost rather than merely minimizing the expected cost. In fact, in real-world decision-making under uncertainty individuals typically prefer more predictable or less uncertain outcomes over those with greater volatility, even if both have the same expected return. This risk averse preferences can be reflected by risk measures which disproportionately penalize outcomes that involve high cost or variability. 

Risk measures typically involve integrals over the possibly infinite-dimensional space of uncertain parameters. The significant computational expense associated with these problems has spurred research into the development of efficient numerical algorithms, such as those based on sparse grids~\cite{kouriheinken, CVaR} or quasi-Monte Carlo (QMC) methods~\cite{GKKSS21, GKKSS24, GKK24, Milz2024}.
The entropic risk measure is such a risk averse measure, 
which can be effectively computed using QMC methods, see~\cite{GKKSS24,LongoSchwabStein}.

In this paper the application of sequential second-order Taylor approximations for the efficient approximation of a risk-averse feedback control law based on the entropic risk measure is investigated. To do so, given a possibly countably infinite sequence of uncertain parameters~$\bssigma = (\sigma)_{j\ge 1} \in \mathfrak{S} \subseteq \bbR^\bbN$, let us consider the parameter-dependent optimal control problem 
\begin{align}
\begin{split}\label{eq:J}
    &\min_{u,y} \mathcal{J}(u,y),\quad \mathcal{J}(u,y) := \frac12\bigg(  \int_0^T\left( \mathcal{R}\left(\|C(y(\cdot;t) - g(t))\|_H^2\right) + \|u(t)\|_{U}^2\right) \mathrm dt \\
    &\qquad\qquad\qquad\qquad\qquad\qquad\qquad\qquad + \mathcal{R}\left(\|P(y(\cdot;T) - g_T)\|_H^2 \right)\bigg),
\end{split}
\end{align}
subject to 
\begin{align}\label{eq:sys}
    \dot{y}(\bssigma;t) = A(\bssigma) y(\bssigma;t) + B u(t) + f(t) \quad y(\bssigma;0) = y_0, 
\end{align}
for all~$\bssigma \in \mathfrak{S}$. Given~$T>0$, the goal is to find a control input~$u \in L^2(0,T;U)$ which steers the parameter-dependent state~$y(\bssigma)$ as close as possible to the targets~$g \in \mathcal{C}([0,T];H)$ and~$g_T \in H$. The initial condition of the state~$y(\bssigma;0) = y_0 \in H$ is known. The time evolution of the system is described by the parameter-dependent operators~$A(\bssigma) \in \mathcal{L}(V,V')$ acting on the current state,~$B\in \mathcal{L}(U,H)$ acting on the input control, and an external forcing~$f\in L^2(0,T;V')$. The operator~$C\in \mathcal{L}(H)$ is an observation operator and~$P \in \mathcal{L}(H)$. In principle,~$\mathcal{R}$ can denote any risk measure. However, in this work, we focus on the entropic risk measure since it is smooth and it can be computed effectively for high-dimensional uncertain parameters.

\paragraph{Uncertain parameters}
Throughout this manuscript we assume that the domain of the uncertain variables is given as
\begin{align*}
    \mathfrak{S} := \left[-1,1\right]^\bbN.
\end{align*}
Furthermore, the components~$\sigma_j$ of the parameter sequence~$\bssigma  = (\sigma_j)_{j\in \bbN}$ are independently and identically distributed random variables, each being uniformly distributed over~$\left[-1,1\right]$. That is, the sequence~$\bssigma$ is distributed according to the countable product probability measure
\begin{align*}
    \mathrm d\bssigma := \bigotimes_{j\in \bbN} \frac{\mathrm d\sigma_j}{2}.
\end{align*}
We shall be interested in continuous functions~$\bssigma \mapsto X(\bssigma)$ taking values in some Banach space. In this case~$X$ is Bochner-integrable and its expected value can be written as infinite-dimensional integral
\begin{align}\label{eq:exp_int}
    \mathbb{E}\left[X\right] = \int_{\mathfrak{S}} X(\bssigma) \, \mathrm d\bssigma.
\end{align}

\paragraph{Entropic risk measure}
We shall focus on a particular choice of the risk measure~$\mathcal{R}$: the entropic risk measure with risk-aversion parameter~$\theta>0$. For a random variable~$X \in L^\infty(\mathfrak{S};\bbR)$, it is defined as
\begin{align*}
    \mathcal{R}_\theta (X) := \frac{1}{\theta} \log{\left( \mathbb{E} \left[ \mathrm e^{\theta X }\right]\right)}.
\end{align*}
As a function of~$\theta$, the entropic risk measure is increasing and strictly increasing if~$X$ is not constant (a.s.). Moreover, it can be shown (see, e.g.,~\cite[Thm.~1.3.2]{modernrisk}) that 
\begin{align*}
    \lim_{\theta \to 0} \mathcal{R}_\theta(X) = \mathbb{E}\left[X\right] \qquad \text{and} \qquad \lim_{\theta \to \infty} \mathcal{R}_\theta(X) = {\rm ess\,sup}(X).
\end{align*}

Due to its exponential form, the entropic risk measure in~\eqref{eq:J} assigns greater weight to scenarios that involve large tracking cost, thereby encouraging control actions that avoid these risks, even if the control cost might be higher. In large deviation theory the entropic risk measure is also called logarithmic moment generating function~\cite{dembo2010large}.

\subsection{Notation}
Given real numbers~$r<s$ and separable Banach spaces~$\mathcal{X}$ and~$\mathcal{Y}$, the space of continuous functions from~$[r,s]$ into~$\mathcal{X}$ is denoted by~$\calC([r,s];\mathcal{X})$ and the Bochner space of strongly measurable square integrable functions from the interval~$(r,s)$ into~$\mathcal{X}$ is denoted by~$L^2(r,s;\mathcal{X})$ and we also denote the subspace~$W(r,s;\mathcal{X},\mathcal{Y}):= \{v \in L^2(r,s;\mathcal{X})\,\mid\,\dot{v} \in L^2(r,s;\mathcal{Y})\}$.
Since the time horizon~$T>0$ will be fixed throughout this manuscript, to shorten the exposition, sometimes we shall denote
\begin{equation*}
{\mathcal{X}}_T := L^2( 0,T;\mathcal{X})\quad\mbox{and}\quad W_T(\mathcal{X},\mathcal{Y}):= W(0,T; \mathcal{X},\mathcal{Y}).
\end{equation*}

By~$\calL(\mathcal{X},\mathcal{Y})$ we denote the space of linear continuous mappings from~$\mathcal{X}$ into~$\mathcal{Y}$, and in case~$\mathcal{X}=\mathcal{Y}$ we use the shorter~$\calL(\mathcal{X}):=\calL(\mathcal{X},\mathcal{X})$.

Throughout this manuscript, boldfaced symbols are used to denote multi-indices while the subscript notation~$m_j$ is used to refer to the~$j^{\rm th}$ component of a multi-index~$\boldsymbol{m}$. Further, let 
\begin{align*}
    \mathscr{F} := \{ \bsm \in \bbN_0^\bbN \mid |\bsm| < \infty \}
\end{align*}
denote the set of finitely supported multi-indices, where the order of a multi-index~$\bsm$ is defined as $|\bsm| := \sum_{j\ge 1} m_j$. Further, we write~$\delta_{\bsm,\boldsymbol{0}} =1$ if~$m_j = 0$ $\forall j \ge1$ and~$\delta_{\bsm,\boldsymbol{0}} = 0$ otherwise.

\section{Problem formulation}\label{sec:probl_form}

In order to show existence and uniqueness of a solution to~\eqref{eq:J} subject to~\eqref{eq:sys}, we make the following assumptions: the parameter-dependent operator~$A(\bssigma)$ in~\eqref{eq:sys} can be associated with a continuous and~$V$-$H$-coercive parameter-dependent bilinear form~$a(\bssigma;\cdot,\cdot)$, that is
\begin{align}\label{eq:A1}
    \langle A(\bssigma) v,w\rangle_{V',V} := -a(\bssigma;v,w),\qquad \forall v,w \in V,\quad \forall \bssigma \in \mathfrak{S},
\end{align}
where~$\exists (\rho,\theta) \in \bbR \times (0,\infty)$ such that
\begin{align}\label{eq:A2}
    a(\bssigma;v,v) + \rho \|v\|_H^2 \ge \theta \|v\|_V^2\quad \forall v \in V\quad {\rm and}\quad \forall \bssigma \in \mathfrak{S} .
\end{align}
In addition we assume that the operators have a uniform upper bound, that is
\begin{align}\label{eq:A3}
    \|A(\bssigma)\|_{\calL(V,V')}\le C_D\quad \forall \bssigma \in \mathfrak{S}.
\end{align}
For every~$\bssigma \in \mathfrak{S}$, let us define the parameterized parabolic evolution operator~$D(\bssigma)$ as
\begin{align}\label{eq:defpara}
    \langle D(\bssigma) w, (v_1,v_2) \rangle_{V'_T \times H,V_T \times H} := \langle \dot w, v_1 \rangle_{V_T',V_T} - \langle A(\bssigma) w, v_1 \rangle_{V_T',V_T} + \langle w(0), v_2 \rangle_H
\end{align}
for all~$w\in W_T(V,V')$ and all~$v=(v_1,v_2) \in V_T \times H$, where we recall that~$V_T = L^2(0,T;V)$.
The parameterized family of operators~$\{D(\bssigma) \in \calL(W_T(V,V'),V'_T \times H): \bssigma \in \mathfrak{S}\}$ has uniformly bounded inverses, see, e.g.,~\cite{GKK24,kunoth2013analytic}.

Throughout the manuscript it will be assumed that 
\begin{align*}
    A:\mathfrak{S} \to \calL(V,V') ,\quad \bssigma \mapsto A(\bssigma) \quad \text{is continuous},
\end{align*}
which leads to the following regularity result for solutions of~\eqref{eq:sys}.
\begin{lemma}\label{lem:extra_reg}
    Given~$u\in U_T$,~$y_0 \in H$, and~$f \in V_T'$, there is a unique solution~$y \in \calC([0,T];\calC(\mathfrak{S};H))$ of~\eqref{eq:sys}.
\end{lemma}
\begin{proof}
    Under the assumption that~$\bssigma \mapsto A(\bssigma) \in \calL(V,V')$ is continuous, the parametric mapping to the evolution operator~$\bssigma \mapsto D(\bssigma) \in \calL(W_T(V,V'),L^2_T(V')\times H)$, as defined in~\eqref{eq:defpara}, is continuous. Then, with~\eqref{eq:A2} and~\eqref{eq:A3} it follows that~$D(\bssigma)$ is boundedly invertible for all~$\bssigma \in \mathfrak{S}$, and hence from \cite[Thm.~1.1.1]{Gittelson2011} we conclude that~$\bssigma \mapsto y \in W_T(V,V')$ is continuous. From the continuous embedding of~$W_T(V,V') \hookrightarrow \calC([0,T];H)$ we have~$y\in \calC(\mathfrak{S};\calC([0,T];H))$, and in particular~$y \in C(\mathfrak{S}\times [0,T];H)$, and consequently~$y(\cdot;t,\cdot) \in \calC(\mathfrak{S};H)$ for every~$t \in [0,T]$. 
\end{proof}

By substituting~$y(\bssigma;t) = D(\bssigma)^{-1}(Bu(t)+f(t),y_0)$ in~\eqref{eq:J} one obtains the reduced formulation of the optimal control problem~$\min_{u} J(u)$, which only depends on the control input~$u$. In the remainder of this section we will prove the existence and uniqueness of a minimizer of the reduced (and equivalently the original) problem and characterize it by a necessary and sufficient optimality condition.

It will sometimes be convenient to write~\eqref{eq:sys} in the equivalent form 
\begin{align*}
    D(\bssigma)y(\bssigma) = (D_1(\bssigma)y(\bssigma), D_2(\bssigma)y(\bssigma)) = (Bu + f, y_0) \quad \text{in } V_T' \times H,
\end{align*} where we define
\begin{align*}
D_1(\bssigma) &:= \Lambda_1 D(\bssigma),\quad \Lambda_1: V_T' \times H \to V_T',\quad \Lambda_1(v_1,v_2) := v_1, \quad\text{for all~$v = (v_1,v_2) \in V_T' \times H$},\\
D_2(\bssigma) &:= \Lambda_2 D(\bssigma),\quad \Lambda_2: V_T' \times H \to H,\quad\,\, \Lambda_2(v_1,v_2) := v_2, \quad\text{for all~$v = (v_1,v_2) \in V_T' \times H$}.
\end{align*}
Meaningful (right-)inverse operators of~$D_1(\bssigma):W_T(V,V') \to V_T'$ and~$D_2(\bssigma):W_T(V,V') \to H$ are given by
\begin{align*}
    D_1^\dagger(\bssigma) &:= D(\bssigma)^{-1} \Xi_1, \quad \Xi_1: V_T' \to V_T'\times H,\quad \Xi_1 v_1 := (v_1,0),\quad\text{for all~$v_1\in V_T'$},\\
    D_2^\dagger(\bssigma) &:= D(\bssigma)^{-1} \Xi_2, \quad \Xi_2: H \to V_T'\times H,\quad \,\,\Xi_2 v_2 := (0,v_2),\quad\text{for all~$v_2\in H$}.
\end{align*}
It follows that, for all~$\bssigma \in \mathfrak{S}$, the unique solution of~\eqref{eq:sys} can be written as
\begin{align*}
    y(\bssigma) = D(\bssigma)^{-1} (Bu + f,y_0) = D_1^\dagger(\bssigma) (Bu + f) + D_2^\dagger(\bssigma) y_0 \quad \text{in } W_T(V,V').
\end{align*}
By substituting this into the objective function~\eqref{eq:J}, we arrive at the reduced problem
\begin{align*}
    \min_{u} J(u), \quad J(u) := \mathcal{J}(u,D(\bssigma)^{-1} (Bu + f,y_0)).
\end{align*}
We show next that this reduced problem (and hence the original one) has a unique solution.

\paragraph{Existence of solutions}
We have that~$0 \le J(u)$, thus there is a sequence~$(u_k)_{k\in \bbN}$ such that~$\lim_{k \to \infty}J(u_k) = \inf_{u} J(u)$. Since~$J(u)$ is coercive ($J(u) \ge \frac12 \int_0^T\|u\|_U^2\,\mathrm dt$), the infimizing sequence is bounded. The sequence takes values in~$L^2(0,T;U)$ and thus it has a weakly convergent subsequence~$u_\ell \rightharpoonup \bar{u}$. Since~$J$ is convex (see also next paragraph) and continuous, it is weakly lower semicontinuous and we have
\begin{align*}
    J(\bar{u}) \le \liminf_{\ell \to \infty} J(u_\ell) = \inf_{u} J(u).
\end{align*}
Thus,~$J(\bar{u}) = \inf_u J(u)$ and~$\bar{u}$ is the sought minimizer.

\paragraph{Uniqueness of the solution}
First, observe that the two terms in the cost functional that involve the entropic risk measure,~$\mathcal{R}_\theta(\|C(y(\cdot;t) - g(t))\|_H^2$ and~$\mathcal{R}_\theta(\|P(y(\cdot;T)- g_T)\|_H^2$, are well-defined for solutions~$y\in \mathcal{C}([0,T];\mathcal{C}(\mathfrak{S};H))$ of~\eqref{eq:sys}. Given~$\theta \ge 0$, for every fixed~$t \in [0,T]$, the functional
\begin{align*}
    u\mapsto \mathcal{R}_\theta(\|C(y(\cdot;t) - g(t))\|_H^2)
\end{align*}
is convex. This follows from the fact, that~$E_t: W_T(V,V') \to H$, which evaluates functions pointwise, is a bounded linear operator, and from the linear dependence of~$y(\bssigma)$ on~$u$, the convexity of~$z \mapsto \|z\|^2_H$, as well as the convexity and the monotonicity of~$\mathcal{R}_\theta$. By the same reasoning the convexity of
\begin{align*}
    u \mapsto \mathcal{R}_\theta(\|P(y(\cdot;T)- g_T)\|_H^2 )
\end{align*}
follows. Furthermore, we have that
$$u \mapsto \int_0^T  \mathcal{R}_\theta(\|C(y(\cdot;t) - g(t))\|_H^2)\, \mathrm dt$$ 
is convex. 
The control cost~$u \mapsto \frac12\int_0^T\|u(t)\|^2_U\,\mathrm dt$ is strongly convex. Thus,~$J(u)$ is strongly convex since it is the positive linear combination of two convex functions and a strongly convex function. Hence, the solution is unique.

\paragraph{Optimality condition}
Since the problem is strongly convex, first-order optimality conditions are necessary and sufficient. Thus, we shall be interested in the gradient of the reduced objective function
\begin{align*}
    J(u) &= \frac12\left(  \int_0^T\left( \frac{1}{\theta} \log{\left( \mathbb{E} \left[ \mathrm e^{\theta \|C(D(\cdot)^{-1}\left((Bu(t)+f(t)),y_0\right) - g(t))\|_H^2  }\right]\right)} +\|u(t)\|_{U}^2\right) \mathrm dt \right.\\ 
    &\quad+\left. \frac{1}{\theta} \log{\left( \mathbb{E} \left[ \mathrm e^{\theta \|P(E_T D(\cdot)^{-1}\left((Bu(t)+f(t)),y_0\right) - g_T)\|_H^2  }\right]\right)}\right).
\end{align*}
Given~$f \in V_T'$ and~$ y_0 \in H$, let~$y_u(\bssigma)$ denote the solution to~\eqref{eq:sys} for a fixed~$\bssigma \in \mathfrak{S}$, and a given input control~$u$. Then, we have that~$y_{u+\delta}(\bssigma) - y_u(\bssigma) = D(\bssigma)^{-1}(B(u+\delta) + f, y_0) - D(\bssigma)^{-1}(Bu + f, y_0) = D(\bssigma)^{-1}(B\delta,0) = D_1^\dagger(\bssigma)B\delta$. Using this result, the Fr\'echet differentiability of~$J$ at~$u \in U_T = L^2(0,T;U)$ follows from the differentiability of both the entropic risk measure~\cite{entropic} and the squared norm, in combination with the chain rule and the Lebesgue dominated convergence theorem.

The gradient~$\nabla J(u) \in U_T$ is identified with the Riesz representer of the Fr\'echet derivative at~$u \in U_T$, that is~$\frac{\partial}{\partial u} J(u) \delta = \langle \nabla J(u), \delta \rangle_{U_T}$ for all directions~$\delta \in U_T$, and given by
\begin{align}
\begin{split}\label{eq:gradient}
    \nabla J(u) =  u &+ \frac{\mathbb{E}\left[\mathrm e^{\theta\|C(y(\cdot)-g\|^2_H} B^\ast (D_1^\dagger(\cdot))^\ast \left(  C^\ast C ( y(\cdot) - g\right)\right]}{\mathbb{E}\left[\mathrm e^{\theta\|C(y(\cdot)-g\|_H^2}\right]} \\&+\frac{\mathbb{E}\left[\mathrm e^{\theta\|P(E_Ty(\cdot)-g_T)\|_H^2} B^\ast (D_1^\dagger(\cdot))^\ast \left(  E_T^\ast P^\ast P (E_T y(\cdot) - g_T)\right)\right]}{\mathbb{E}\left[\mathrm e^{\theta\| P(E_Ty(\cdot)-g_T)\|_H^2}\right]},
\end{split}
\end{align}
where~$E_T^\ast$ denotes the adjoint of the operator~$E_T$.

\section{Quadratic approximation of the entropic risk measure}\label{sec:section3}

Next we explore a quadratic approximation of the entropic risk measure leading to a linear quadratic optimal control problem. This approximation forms the basis of a sequential quadratic programming (SQP) method, which is employed to iteratively approximate the solution of the original optimal control problem involving the nonlinear entropic risk measure, that is~\eqref{eq:J} subject to~\eqref{eq:sys}.

To rigorously establish the Fr\'echet derivatives required for the quadratic approximation, we first revisit the regularity properties of the state~$y$. Under the conditions of Lemma~\ref{lem:extra_reg} we have~$y \in L^2(\mathfrak{S};V_T)$ and~$\dot{y} \in L^2(\mathfrak{S};V'_T)$. Moreover,~$V_T$ as well as~$L^2(\mathfrak{S})$ are separable Hilbert spaces, and hence it holds that~$L^2(\mathfrak{S};V_T) = L^2(\mathfrak{S}) \otimes V_T$, where~$\otimes$ denotes the Hilbert tensor product, where equality (or canonical identification) is up to a unique isometric isomorphism, see~\cite[Thm. II.10]{reedsimon}. Similarly, we have~$V_T = L^2(0,T) \otimes V$ since both~$L^2(0,T)$ and~$V$ are separable. Thus, we have
\begin{align}
\begin{split}\label{eq:swapspaces}
    L^2(\mathfrak{S};V_T) &= L^2(\mathfrak{S}) \otimes (L^2(0,T) \otimes V) = (L^2(\mathfrak{S}) \otimes L^2(0,T)) \otimes V\\
    &=L^2(\mathfrak{S} \times [0,T]) \otimes V = (L^2(0,T) \otimes L^2(\mathfrak{S})) \otimes V\\ 
    &=  L^2(0,T) \otimes (L^2(\mathfrak{S}) \otimes V) = L^2(0,T;L^2(\mathfrak{S};V)),
    \end{split}
\end{align}
where we used the associativity (up to unique isometric ismorphism) of the tensor product, see~\cite[Prop. 2.6.5.]{kadisonringrose}. Analogously, we have~$\dot{y} \in L^2(0,T;L^2(\mathfrak{S};V'))$, and consequently~$y \in W_T(L^2(\mathfrak{S};V),L^2(\mathfrak{S};V'))$.

To compute Fr\'echet derivatives, we define 
\begin{align*}
    Z := \{(u,y) \in U_T \times W_T(L^2(\mathfrak{S};V),L^2(\mathfrak{S};V')) \mid \dot{y} = \mathcal{A}y + Bu + f,\, y(0) = y_0\},
\end{align*}
where~$\mathcal{A} \in \calL(L^2(\mathfrak{S};V),L^2(\mathfrak{S};V'))$ is the extension of the pointwise operator family~$\{A(\bssigma) \in \calL(V,V')\mid \bssigma \in \mathfrak{S}\}$ to the Lebesgue--Bochner space as~$\mathcal{A} v := [\bssigma \mapsto A(\bssigma) v(\bssigma)]$. The dynamical system in the definition of~$Z$ is posed in~$L^2(0,T;L^2(\mathfrak{S};V'))$, which is identified with~$L^2(\mathfrak{S};V_T')$ by~\eqref{eq:swapspaces}, with~$V_T$ replaced by~$V_T'$. Thus, the dynamics in~$Z$ can be recast as equation in~$V_T'$ as follows:
\begin{align}\label{eq:aesol}
    \dot{y}(\bssigma) = A(\bssigma) y(\bssigma) + B u + f, \quad y(\bssigma;0,\cdot) = y_0
\end{align}
for almost every~$\bssigma \in \mathfrak{S}$. On the other hand, given~$f$,~$u$, and~$y_0$, the system~\eqref{eq:aesol} has a unique solution~$\tilde{y}(\bssigma) \in W_T(V,V')$ for every~$\bssigma \in \mathfrak{S}$, which by Lemma~\ref{lem:extra_reg} admits the additional regularity~$\tilde{y} \in \calC(\mathfrak{S};\calC([0,T];H))$. For the solution~$y$ to~\eqref{eq:aesol} we find that~$y = \tilde{y}$ up to~$\mathrm d\bssigma$-almost everywhere equivalence. We have thus shown the following result:
\begin{lemma}
    Given~$f \in V_T'$ and~$y_0 \in H$, the state component of tuples~$(u,y) \in Z$ admits a representative~$y \in \calC(\mathfrak{S};\calC([0,T]; H))$.
\end{lemma}
Hence, in the following, we shall make use of~$y \in  \calC(\mathfrak{S};\calC([0,T];H))$ for elements~$y\in Z$ wherever continuity or pointwise evaluation is required.

Let~$\bar{y} \in \mathcal{C}([0,T];\mathcal{C}(\mathfrak{S};H))$, for example let~$\bar{y}$ be the solution to~\eqref{eq:sys} corresponding to some control~$\bar{u} \in U_T$. A quadratic approximation about~$\bar{y}$ of the cost functional~\eqref{eq:J}, is of the form
\begin{align}
\begin{split}\label{eq:quadJ}
    &\mathcal{J}_{\rm quad}(u,y) \\
    &= \frac{1}{2}\bigg( \int_0^T \Big(  \|u(t)\|^2_U + \mathcal{R}_\theta(\|C(\bar{y}(\cdot;t)-g(t))\|_H^2) \\
    &\quad+ \Big\langle \frac{\partial}{\partial y}\mathcal{R}_\theta(\|C(y(\cdot;t))-g(t))\|_H^2)\big|_{y = \bar{y}(\cdot;t)}, y(\cdot;t)- \bar{y}(\cdot;t)\Big\rangle_{L^2(\mathfrak{S};H)} \\
    &\quad+ \frac{1}{2} \left[\frac{\partial^2}{\partial y^2} \mathcal{R}_\theta(\|C(y(\cdot;t)-g(t))\|_H^2)\big|_{y = \bar{y}(\cdot;t)}\right] \left(y(\cdot;t)- \bar{y}(\cdot;t), y(\cdot;t)- \bar{y}(\cdot;t)\right) \Big) \mathrm dt\\
    &\quad+ \mathcal{R}_\theta(\|P(E_T\bar{y}(\cdot;t)-g_T)\|_H^2) \\
    &\quad+ \int_0^T \Big(\Big\langle \frac{\partial}{\partial y}\mathcal{R}_\theta(\|P(E_Ty(\cdot;t)-g_T)\|_H^2)\big|_{y = \bar{y}(\cdot;t)}, y(\cdot;t)- \bar{y}(\cdot;t)\Big\rangle_{L^2(\mathfrak{S};H)} \\
    &\quad+ \frac{1}{2}\left[ \frac{\partial^2}{\partial y^2} \mathcal{R}_\theta(\|P(E_Ty(\cdot;t)-g_T)\|_H^2)\big|_{y = \bar{y}(\cdot;t)} \right]\left(y(\cdot;t)- \bar{y}(\cdot;t), y(\cdot;t)- \bar{y}(\cdot;t)\right) \Big)\mathrm dt\bigg),
\end{split}
\end{align}
for~$(u,y) \in U_T \times \mathcal{C}([0,T];\mathcal{C}(\mathfrak{S};H))$.

For the remainder of the manuscript it will be useful to define 
\begin{align*}
    \mathscr{C}(\bssigma;t) := C^\ast C (y(\bssigma;t) -g(t))\quad \text{and} \quad \mathscr{P}(\bssigma;T) := P^\ast P (y(\bssigma;T) -g_T),
\end{align*}
as well as the weight functions
\begin{align*}
    \omega_{\theta,C}(y(\bssigma;t)) := \frac{\mathrm e^{\theta \|C(y(\bssigma;t)-g(t))\|_H^2}}{\mathbb{E} \left[\mathrm e^{\theta \|C(y(\cdot;t)-g(t))\|_H^2} \right]}, \qquad \text{and}  \qquad \omega_{\theta,P}(y(\bssigma;T)) = \frac{\mathrm e^{\theta \|P( y(\bssigma;T)-g_T)\|_H^2}}{\mathbb{E} \left[\mathrm e^{\theta \|P(y(\cdot;T)-g_T)\|_H^2} \right]},
\end{align*}
which are well-defined for~$y \in \mathcal{C}([0,T];\mathcal{C}(\mathfrak{S};H))$. Furthermore, for an integrable element~$f$ taking values in a separable Banach space, we define the weighted expectations
\begin{align*}
    \mathbb{E}_{\omega_{\theta,C}(y(\bssigma;t))}\left[f\right] := \mathbb{E}\left[f(\cdot)\,  \omega_{\theta,C}(y(\cdot;t)) \right] \quad \text{and} \quad\mathbb{E}_{\omega_{\theta,P}(y(\bssigma;T))}\left[f\right] := \mathbb{E}\left[f(\cdot)\,  \omega_{\theta,P}(y(\cdot;T)) \right]. 
\end{align*}
Using this notation, the first-order Fr\'echet derivatives at~$y \in \mathcal{C}([0,T];\mathcal{C}(\mathfrak{S};H))$ in direction~$\delta_1 \in \mathcal{C}([0,T];\mathcal{C}(\mathfrak{S};H))$ can be expressed as
\begin{align*}
&\int_0^T \left\langle \frac{\partial}{\partial y} \mathcal{R}_{\theta}(\|C(y(\cdot,t)-g(t))\|_H^2), \delta_1(\cdot;t)\right\rangle_{L^2(\mathfrak{S};H)} \mathrm dt \\
&\quad\quad\quad\quad\quad\quad\quad\quad= 2 \int_0^T \frac{\mathbb{E}[\mathrm e^{\theta \|C(y(\cdot;t)-g(t))\|_H^2 } \langle C^\ast C(y(\cdot;t)-g(t)),\delta_1(\cdot;t)\rangle_H]}{\mathbb{E}[\mathrm e^{\theta \|C(y(\cdot;t)-g(t)) \|_H^2}]} \,\mathrm dt\\
&\quad\quad\quad\quad\quad\quad\quad\quad= 2 \int_0^T \mathbb{E}_{\omega_{\theta,C}(y(\bssigma;t))}\left[\langle \mathscr{C}(\cdot;t),\delta_1(\cdot;t)\rangle_H\right] \, \mathrm dt\\
&\int_0^T \left\langle \frac{\partial}{\partial y} \mathcal{R}_{\theta}(\|P(E_Ty(\cdot;t)-g_T)\|_H^2), \delta_1(\cdot;t)\right\rangle_{L^2(\mathfrak{S};H)} \mathrm dt \\
&\quad\quad\quad\quad\quad\quad\quad\quad= 2 \int_0^T \frac{\mathbb{E}[\mathrm e^{\theta \|P(E_Ty(\cdot;t)-g_T)\|_H^2 } \langle E_T^\ast P^\ast P(E_Ty(\cdot;t)-g_T),\delta_1(\cdot;t)\rangle_H]}{\mathbb{E}[\mathrm e^{\theta \|P(E_Ty(\cdot;t)-g_T) \|_H^2}]} \,\mathrm dt
\\
&\quad\quad\quad\quad\quad\quad\quad\quad= 2 \,\mathbb{E}_{\omega_{\theta,P}(y(\bssigma;T))} \left[  \langle \mathscr{P}(\cdot;T),\delta_1(\cdot;T)\rangle_H\right].
\end{align*}%
This follows again from the chain rule and Lebesgue's dominated convergence theorem. We point out that both derivatives can be extended to bounded linear functionals on $L^2(0,T;L^2(\mathfrak{S};H))$. By Lemma~\ref{lem:extra_reg} these derivatives are well-defined on~$Z$, which will be used in Section~\ref{sec:SQP} to prove that the solution of the quadratic approximation converges to the solution of the original problem. Similarly, evaluated on~$Z$, the second-order Fr\'echet derivatives in~\eqref{eq:quadJ} are identified as the bounded bilinear forms
\begin{align*}
&\int_0^T \left[\frac{\partial^2}{\partial y^2}\mathcal{R}_{\theta}(\|C(y(\cdot;t)-g(t))\|^2_H)\right] (\delta_2(\cdot;t),\delta_1(\cdot;t)) \,\mathrm dt\\
&=\int_0^T 2\,\mathbb{E}_{\omega_{\theta,C}(y(\bssigma;t))} \Big[ \langle C^\ast C \delta_2(\cdot;t), \delta_1(\cdot;t) \rangle_H \Big]\\
&\quad+ 4 \theta\, \mathbb{E}_{\omega_{\theta,C}(y(\bssigma;t))} \Big[ \langle \mathscr{C}(\cdot;t), \delta_2(\cdot;t) \rangle_H\, \langle \mathscr{C}(\cdot;t), \delta_1(\cdot;t) \rangle_H\Big]\\
&\quad- 4 \theta\, \mathbb{E}_{\omega_{\theta,C}(y(\bssigma;t))} \Big[ \langle \mathscr{C}(\cdot;t), \delta_1(\cdot;t) \rangle_H\Big] \mathbb{E}_{\omega_{\theta,C}(y(\bssigma;t))} \Big[ \langle \mathscr{C}(\cdot;t), \delta_2(\cdot;t) \rangle_H\Big]\mathrm dt\\
&=\int_0^T 2\,\mathbb{E}_{\omega_{\theta,C}(y(\bssigma;t))} \Big[ \langle C^\ast C \delta_2(\cdot;t), \delta_1(\cdot;t) \rangle_H \Big]\\
&\quad+ 4 \theta\,{\rm Cov}_{\omega_{\theta,C}(y(\bssigma;t))}\Big( \langle \mathscr{C}(\cdot;t), \delta_2(\cdot;t) \rangle_H, \langle \mathscr{C}(\cdot;t), \delta_1(\cdot;t) \rangle_H\Big)\mathrm dt,
\end{align*}
for~$\delta_1 \in L^2(0,T;L^2(\mathfrak{S};H))$ and~$\delta_2 \in L^2(0,T;L^2(\mathfrak{S};H))$, and analogously
\begin{align*}
    &\int_0^T\left[\frac{\partial^2}{\partial y^2}\mathcal{R}_{\theta}(\|P(E_Ty(\cdot;t)-g_T)\|^2_H)\right] (\delta_2(\cdot;t),\delta_1(\cdot;t)) \,\mathrm dt\\
    &= 2\,\mathbb{E}_{\omega_{\theta,P}(y(\bssigma;T))} \left[ \langle P^\ast P \delta_2(\cdot;T), \delta_1(\cdot;T) \rangle_H \right] \\
    &\quad+ 4 \theta\,{\rm Cov}_{\omega_{\theta,P}(y(\bssigma;T))}\Big( \langle \mathscr{P}(\cdot;T), \delta_2(\cdot;T) \rangle_H, \langle \mathscr{P}(\cdot;T), \delta_1(\cdot;T) \rangle_H\Big),
\end{align*}
As a consequence, if~$\delta_1 = \delta_2$, the second derivative is nonnegative and strictly positive if additionally~$\mathscr{C}$ or~$\mathscr{P}$ is not a constant (almost surely). Moreover, the second Fr\'echet derivatives are locally Lipschitz continuous with respect to~$y \in \mathcal{C}([0,T];\mathcal{C}(\mathfrak{S};H))$, which will be used in Section~\ref{sec:SQP}. 

\begin{remark}Several remarks are in order:
    \begin{itemize}
    \item We point out that the second-order Fr\'echet derivative differs from the risk neutral case not only by the additional covariance terms, but also the expected values are weighted by the normalized exponential functions~$\omega_{\theta,C}(y(\bssigma;t))$ and~$\omega_{\theta,P}(y(\bssigma;t))$. These weights reflect the risk averse control decision of the entropic risk measure by assigning exponentially more importance to realizations of the parameter sequence~$\bssigma \in \mathfrak{S}$ which lead to large tracking cost. The exponential reweighting of a probability measure to make rare events more likely is known in large deviation theory (\cite[Ch.~3.1]{bucklew2014introduction}) and rare event simulation~\cite{rubino2009rare} as exponential tilting or exponential twisting.
    \item The weighted expected values are related to expected values with respect to the posterior measure in Bayesian inverse problems~\cite{Stuart_2010}. Specifically, in Bayesian inverse problems with forward model~$\mathcal{G}$ and centered additive Gaussian noise on the data~$Y_{\rm data} = G(\xi) + \eta$,~$\eta \sim \mathcal{N}(0,\Gamma)$, the prior belief on the unknown parameter~$\xi$ is updated by a likelihood of the form $\exp(-\|G(\xi)-Y_{\rm data}\|^2_{\Gamma})/\int \exp(-\|G(\xi)-Y_{\rm data}\|^2_{\Gamma}) \,\mathrm d\mathcal{N}(0,\Gamma)(\xi)$ resulting in a posterior probability distribution of the parameter given the data. This structural similarity between Bayesian inverse problems and optimal control problems with entropic risk measure has been used, e.g., in~\cite{LongoSchwabStein}.
    \end{itemize}
\end{remark}

\subsection{Quadratic approximation leads to risk-adjusted linear quadratic problem}\label{sec:risk-adj}
The second-order Fr\'echet derivatives at~$\bar{y} \in \mathcal{C}([0,T];\mathcal{C}(\mathfrak{S};H))$ are associated with the operators~$\mathcal{Q}_{C,\omega}(\bar{y};t) \in \calL(L^2(\mathfrak{S};H))$ and~$\mathcal{Q}_{P,\omega}(\bar{y};T) \in \calL(L^2(\mathfrak{S};H))$ as
\begin{align*}
    \langle \mathcal{Q}_{C,\omega}(\bar{y}(\cdot;t);t) \delta_1,\delta_2\rangle_{L^2(\mathfrak{S};H)} &=  \mathbb{E}_{\omega_{\theta,C}(\bar{y}(\bssigma;t))} \left[ \langle C^\ast C \delta_2(\cdot;t), \delta_1(\cdot;t) \rangle_H \right]\\
    &\quad + 2\theta\,{\rm Cov}_{\omega_{\theta,C}(y(\bssigma;t))}\Big( \langle \overline{\mathscr{C}}(\cdot;t), \delta_2(\cdot;t) \rangle_H, \langle \overline{\mathscr{C}}(\cdot;t), \delta_1(\cdot;t) \rangle_H\Big)\\
    \intertext{for almost every~$t \in [0,T]$ and}
    \langle \mathcal{Q}_{P,\omega}(\bar{y}(\cdot;T);T) \delta_1, \delta_2 \rangle_{L^2(\mathfrak{S};H)} &= \mathbb{E}_{\omega_{\theta,P}(\bar{y}(\bssigma;T))} \left[ \langle P^\ast P \delta_2(\cdot;T), \delta_1(\cdot;T) \rangle_H \right] \\
    &\quad+2\theta\,{\rm Cov}_{\omega_{\theta,P}(y(\bssigma;T))}\Big( \langle \overline{\mathscr{P}}(\cdot;T), \delta_2(\cdot;T) \rangle_H, \langle \overline{\mathscr{P}}(\cdot;T), \delta_1(\cdot;T) \rangle_H\Big),
\end{align*}
where we use~$\overline{\mathscr{C}}(\bssigma;t) := C^\ast C (\bar{y}(\bssigma;t) -g(t))$ and $\overline{\mathscr{P}}(\bssigma;T) := P^\ast P (\bar{y}(\bssigma;T) -g_T)$. Furthermore, we include the factor~$\frac{1}{2}$ for the second-order derivatives in the Taylor expansion~\eqref{eq:quadJ}. These two operators sum up to the second-order Fr\'echet derivative of the (pointwise in time) tracking error composed with the entropic risk measure, that is
\begin{align}
\begin{split}\label{eq:op_L}
    \langle L(\bar{y}(\cdot;t);t,T) \delta_1, \delta_2 \rangle_{L^2(\mathfrak{S};H)} :=&\, \langle \left(\mathcal{Q}_{C,\omega}(\bar{y}(\cdot;t);t) + \mathcal{Q}_{P,\omega}(\bar{y}(\cdot;T);T)\right) \delta_1,\delta_2 \rangle_{L^2(\mathfrak{S};H)}
\end{split}
\end{align}%
for almost every~$t \in [0,T]$ as well as for all directions $\delta_1 \in L^2(0,T;L^2(\mathfrak{S};H))$ and $\delta_2 \in L^2(0,T;L^2(\mathfrak{S};H))$. Given~$\bar{y} \in \mathcal{C}([0,T];\mathcal{C}(\mathfrak{S};H))$, 
for almost every~$t \in [0,T]$,
the operator $L(\bar{y}(\cdot;t);t,T) \in \calL(L^2(\mathfrak{S};H))$ is nonnegative and self-adjoint. 

Let us assume that~$P$ and~$C$ are not both zero (otherwise~\eqref{eq:J}--\eqref{eq:sys} admits only the trivial solution). Moreover, we assume that there exist positive constants~$c_C$ and~$c_P$ such that~$\|Cx\|_H \ge c_C \|x\|_H$ if~$C \neq 0$ and~$\|Px\|_H \ge c_P \|x\|_H$ for all~$x \in H$ if~$P\neq 0$. Then both operators~$\mathcal{Q}_{C,\omega}(\bar{y}(\cdot;t);t)$ for almost every~$t \in [0,T]$ and~$\mathcal{Q}_{T,\omega}(\bar{y}(\cdot;T);T)$ are coercive (or zero), and thus also~$L(\bar{y}(\bssigma;t);t,T)$ is coercive. Hence, for almost every~$t\in [0,T]$ and any~$\bar{y} \in \mathcal{C}([0,T];\mathcal{C}(\mathfrak{S};H))$, there exist unique positive self-adjoint operators~$\mathcal{Q}_{C,\omega}(\bar{y}(\cdot;t);t)^{1/2}$ and~$\mathcal{Q}_{T,\omega}(\bar{y}(\cdot;T);T)^{1/2}$.

Next, for~$\bar{y} \in \mathcal{C}([0,T];\mathcal{C}(\mathfrak{S};H))$, we set 
\[
\mathscr{R}_C(\bar{y}(\bssigma;t);t):= \overline{\mathscr{C}}(\cdot;t) \omega_{\theta,C}(\bar{y}(\bssigma;t)) ,\quad\text{and}\quad \mathscr{R}_P(\bar{y}(\bssigma;T);T):= \overline{\mathscr{P}}(\cdot;T) \omega_{\theta,P}(\bar{y}(\bssigma;T)),
\]
so that the quadratic approximation~\eqref{eq:quadJ} of the cost functional~\eqref{eq:J} can be written as
\begin{align*}
    \mathcal{J}_{\rm quad}(u,y) &= \frac{1}{2} \bigg(\mathcal{R}_\theta(\|P(\bar{y}(\cdot;T)-g_T)\|_H^2) + 2\, \mathbb{E}\left[ \langle \mathscr{R}_P(\bar{y}(\bssigma;T);T), y(\cdot;T)- \bar{y}(\cdot;T)\rangle_H\right]\\
    &\quad+ \langle \mathcal{Q}_{P,\omega}(\bar{y}(\cdot;T);T) (y(\cdot;T)-\bar{y}(\cdot;T)), (y(\cdot;T) - \bar{y}(\cdot;T))\rangle_{L^2(\mathfrak{S};H)}\\
    &\quad+ \int_0^T \Big( \mathcal{R}_\theta(\|C(\bar{y}(\cdot;t)-g(t))\|_H^2)+ 2\,\mathbb{E}\left[\langle \mathscr{R}_C(\bar{y}(\bssigma;t);t), y(\cdot;t)- \bar{y}(\cdot;t)\rangle_H\right] \\
    &\quad+ \langle \mathcal{Q}_{C,\omega}(\bar{y}(\cdot;t);t) (y(\cdot;t)-\bar{y}(\cdot;t)), (y(\cdot;t) - \bar{y}(\cdot;t))\rangle_{L^2(\mathfrak{S};H)} + \|u(t)\|_U^2\Big) \mathrm dt \bigg).
\end{align*}%
To reformulate the objective function as a sum of quadratic terms only, we substitute~$z_C(\bssigma;t) = \mathcal{Q}_{C,\omega}(\bar{y}(\bssigma;t);t)^{1/2}(y(\bssigma;t)-\bar{y}(\bssigma;t))$ and $z_P(\bssigma;T) = \mathcal{Q}_{P,\omega}(\bar{y}(\bssigma;T);T)^{1/2}(y(\bssigma;T)-\bar{y}(\bssigma;T))$, leading to
\begin{align*}
    &\mathcal{J}_{\rm quad}(u,z) \\
    &\quad= \frac{1}{2} \bigg( \mathcal{R}_\theta(\|P(\bar{y}(\cdot;T)-g_T)\|_H^2) +2\,\mathbb{E}\left[\langle (\mathcal{Q}_{P,\omega}(\bar{y}(\cdot;T);T)^{-1/2})^\ast \mathscr{R}_P(\bar{y}(\cdot;T);T), z_P(\cdot;T)\rangle_{H}\right]\\
    &\quad\quad+ \mathbb{E}\left[\|z_P(\cdot;T)\|_{H}^2\right] + \int_0^T \Big( \|u(t)\|_U^2 + \mathcal{R}_\theta(\|C(\bar{y}(\cdot;t)-g(t))\|_H^2) \\
    &\quad\quad+ 2\,\mathbb{E}\left[\langle (\mathcal{Q}_{C,\omega}(\bar{y}(\cdot;t);t)^{-1/2})^\ast \mathscr{R}_{C}(\bar{y}(\cdot;t);t), z_C(\cdot;t)\rangle_{H}\right] + \mathbb{E}\left[\|z_C(\cdot;t)\|^2_{H}\right] \Big)\mathrm dt \bigg) \\
    &\quad= \frac{1}{2} \bigg( \mathcal{R}_\theta(\|P(\bar{y}(\cdot;T)-g_T)\|_H^2) + \mathbb{E}\left[\|z_P(\cdot;T) + (\mathcal{Q}_{P,\omega}(\bar{y}(\cdot;T);T)^{-1/2})^\ast \mathscr{R}_P(\bar{y}(\cdot;T);T)\|^2_{H}\right]\\
    &\quad\quad - \mathbb{E}\left[\|(\mathcal{Q}_{P,\omega}(\bar{y}(\cdot;T);T)^{-1/2})^\ast \mathscr{R}_{P}(\bar{y}(\cdot;T);T)\|^2_{H} \right] +\int_0^T \Big(\|u(t)\|_U^2 + \mathcal{R}_\theta(\|C(\bar{y}(\cdot;t)-g(t))\|_H^2) \\
    &\quad\quad+ \mathbb{E}\left[\|z_C(\cdot;t) + (\mathcal{Q}_{C,\omega}(\bar{y}(\cdot;t);t)^{-1/2})^\ast \mathscr{R}_{C}(\bar{y}(\cdot;t);t)\|^2_{H}\right] \\
    &\quad\quad- \mathbb{E}\left[\|(\mathcal{Q}_{C,\omega}(\bar{y}(\cdot;t);t)^{-1/2})^\ast \mathscr{R}_{C}(\bar{y}(\cdot;t);t)\|^2_{H}\right] \Big)\mathrm dt\bigg).
\end{align*}%
In the following we will leave out the terms which do not depend on~$u$ or~$y$ since they have no effect on the minimizer of~$\mathcal{J}_{\rm quad}$. This leads to the linear quadratic optimal control problem~$\min_{u,y} J_{\rm quad}$, subject to~\eqref{eq:sys}, where
\begin{align}
\begin{split}\label{eq:LQR}
    J_{\rm quad}(u,y) &= \frac{1}{2} \bigg( \int_0^T \Big( \mathbb{E}\left[\|\mathcal{Q}_{C,\omega}(\bar{y}(\cdot;t))^{1/2}(y(\cdot;t) -  \widetilde{g}(\cdot;t))\|_{H}^2\right]+ \|u(t)\|_U^2 \Big)\,\mathrm dt\\
    &\quad+\mathbb{E}\left[\|\mathcal{Q}_{P,\omega}(\bar{y}(\cdot;T))^{1/2}(y(\cdot;T) -  \widetilde{g}_T(\cdot))\|_{H}^2\right]\bigg)
\end{split}
\end{align}
with target~$\widetilde{g}(\bssigma;t) = \bar{y}(\bssigma;t)-\mathcal{Q}_{C,\omega}(\bar{y}(\bssigma;t))^{-1} \mathscr{R}_C(\bar{y}(\bssigma;t);t)$ and terminal target~$\widetilde{g}_T(\bssigma) = \bar{y}(\bssigma;T)-\mathcal{Q}_{P,\omega}(\bar{y}(\bssigma;T);T)^{-1} \mathscr{R}_P(\bar{y}(\bssigma;T);T)$. For~$\theta \to 0$, the risk-adjusted quadratic objective functional~\eqref{eq:LQR} coincides with the risk-neutral case, that is with~\eqref{eq:J} with~$\mathcal{R} = \mathbb{E}$.

\section{Optimal feedback law of the quadratic approximation}\label{sec:optLQR}
We shall be interested in an optimal feedback law~$\mathscr{K}(\cdot;t): L^2(\mathfrak{S};H) \to U$ for the linear quadratic probem~$\min_{u,y}J_{\rm quad}(u,y)$ subject to~\eqref{eq:sys}.
It will turn out that this feedback law is affine, that is
$$
\mathscr{K}(\cdot;t) = -B^\ast (\Pi(T-t) (\cdot) + h(t)),
$$
where~$\Pi$ solves a Riccati equation. For our main result, we will adapt~\cite[Part~IV, Ch.~2, Thm.~7.3]{bensoussan2007representation} to the setting presented in this work. To this end, we recall the extension of the pointwise operator~$A(\bssigma) \in \mathcal{L}(V,V')$ for~$\bssigma \in \mathfrak{S}$ to Lebegue--Bochner spaces given by
$$\mathcal{A} \in \mathcal{L}(L^2(\mathfrak{S};V),L^2(\mathfrak{S};V')), \quad \mathcal{A} v := [\bssigma \mapsto A(\bssigma) v(\bssigma)].$$ 
Secondly, we introduce the cone~$\Omega(L^2(\mathfrak{S};H))$ of bounded, linear, self-adjoint, and nonnegative operators in~$L^2(\mathfrak{S};H)$ endowed with the norm
of~$\mathcal{L}(L^2(\mathfrak{S};H))$. The Riccati operators will be sought as strongly continuous
operator-valued functions in the set~$\mathcal{S} := \mathcal{C}_s([0,T];\Omega(L^2(\mathfrak{S};H)))$, which is endowed with the topology of strong convergence, i.e., $F_n \to F$ holds if and only if for all $x \in L^2(\mathfrak{S};H)$ it holds that $F_n(\cdot) x \to F(\cdot) x$ in $\mathcal{C}([0,T];L^2(\mathfrak{S};H))$, see e.g.,~\cite[Part IV, Ch.~2.1]{bensoussan2007representation}.

Thirdly, we verify that~$t \mapsto \mathcal{Q}_{C,\omega}(\bar{y}(\cdot;t);t)x \in \mathcal{C}([0,T];L^2(\mathfrak{S};H))$
for any~$x \in L^2(\mathfrak{S};H)$. Indeed, for~$x \in L^2(\mathfrak{S};H)$, we have
\begin{align*}
\mathcal{Q}_{C,\omega}(\bar{y}(\cdot;t);t)x  
&=\frac{\mathrm e^{\theta \|C(\bar{y}(\cdot;t)-g(t))\|_H^2}}{\mathbb{E} \left[\mathrm e^{\theta \|C(\bar{y}(\cdot;t)-g(t))\|_H^2} \right]} \bigg( 2\, C^\ast Cx  + 4 \theta\, \overline{\mathscr{C}}(\cdot;t)  \langle \overline{\mathscr{C}}(\cdot;t), x \rangle_{H} \\
&\!\!\!\!\!\!\!\!\!\!\!\!\!\!\!\!\!- 4 \theta\, \overline{\mathscr{C}}(\cdot;t)\mathbb{E}_{\omega_{\theta,C}(\bar{y}(\bssigma;t))} \Big[ \langle \overline{\mathscr{C}}(\cdot;t), x \rangle_{H}\Big] \bigg) \in \mathcal{C}([0,T];L^2(\mathfrak{S};H)).
\end{align*}

Finally, the above mentioned result~\cite[Part~IV, Ch.~2, Thm.~7.3]{bensoussan2007representation} is stated for the case~$\widetilde{g} = \widetilde{g}_T = 0$. A generalization to tracking problems with nontrivial targets can be found, e.g., in~\cite[Ch.~III, Eqn.~(4.88)]{Lions1}. The latter result can readily be generalized to nontrivial terminal penalizations (as in~\cite[Part~IV, Ch.~2, Thm.~7.3]{bensoussan2007representation}). Thus, we arrive at the following result.

\begin{theorem}\label{thm:main}
The problem~$\min_{u,y} J_{\rm quad}(u,y)$ subject to~\eqref{eq:sys} has a unique optimal pair~$(u,y)$, and the optimal control~$u \in \mathcal{C}([0,T];U)$ is related to the optimal state by the feedback formula
\begin{align*}
    u(t) = -B^\ast ( \Pi(T-t) y(\cdot;t) + h(t)),
\end{align*}
for~$t\in [0,T]$, where~$\Pi \in \mathcal{S}$ and~$h \in W_T(L^2(\mathfrak{S};V);L^2(\mathfrak{S};V'))$ solve
\begin{align*}
    -\dot{\Pi}(T-t) &= \Pi(T-t) \mathcal{A} + \mathcal{A}^\ast \Pi(T-t) - \Pi(T-t) B B^\ast \Pi(T-t) + \mathcal{Q}_{C,\omega}(\bar{y}(\cdot;t);t),\\
    -\dot{h}(t) &=  (\mathcal{A}^\ast - \Pi(T-t) B B^\ast) h(t) + \Pi(T-t) f(t) - \mathcal{Q}_{C,\omega}(\bar{y}(\cdot;t);t) \widetilde{g}(\cdot;t;\bar{y}(\cdot;t)),
\end{align*}
with~$\Pi(0) = \mathcal{Q}_{P,\omega}(\bar{y}(\cdot;T);T)$ and~$ h(T) = - \mathcal{Q}_{P,\omega}(\bar{y}(\cdot;T);T) \widetilde{g}_T(\cdot,\bar{y}(\cdot;T))$. 
\end{theorem}

\section{Convergence of the quadratic approximation}\label{sec:SQP}

Recall that the unique optimal pair of~$\min_{u,y}J_{\rm quad}$ subject to~\eqref{eq:sys} coincides with the unique optimal pair of~$\min_{u,y}\mathcal{J}_{\rm quad}$ subject to~\eqref{eq:sys} since the two cost functionals differ only by a term which is independent of both the control~$u$ and state~$y$.

We shall show next, that the sequence of minimizers~$\{(u^{(k)}_\star,y^{(k)}_\star)\}_{k \ge 0}$ generated by repeatedly solving the linear quadratic problem with updated expansion points~$\bar{y}^{(k)} = y^{(k-1)}_\star$ converges with second order to the unique minimizer of the original problem~\eqref{eq:J}--\eqref{eq:sys} provided that the initial guess is sufficiently close to the minimizer, as will be made precise below.

To this end, we show that the optimality system of the quadratic approximation $\min_{u,y}\mathcal{J}_{\rm quad}$ subject to~\eqref{eq:sys} is a Newton step for the original problem~\eqref{eq:J}--\eqref{eq:sys}. Thus, repeatedly solving the quadratic approximation with updated expansion points results in a so-called sequential quadratic programming (SQP) method for solving the original problem.
The SQP method is based on the Lagrangian
$$
\mathscr{L}(u,y,q) = \mathcal{J}(u,y) + \langle q, e(u,y)\rangle_{L^2(0,T;L^2(\mathfrak{S};V))\times L^2(\mathfrak{S};H),L^2(0,T;L^2(\mathfrak{S};V'))\times L^2(\mathfrak{S};H)}.
$$
where we abbreviated~$e(u,y) = D(\cdot) y(\cdot) - (Bu+f,y_0)$. Let~$q_\star = q(y_\star)$ denote the Lagrange multiplier at the solution~$y_\star$ of~\eqref{eq:J}--\eqref{eq:sys}, then a first-order necessary optimality condition is 
\begin{align}\label{eq:FOCLagrange}
    \mathscr{L}'(u_\star,y_\star,q_\star) = 0,\quad e(u,y) = 0,
\end{align}
where~$\mathscr{L}'$ denotes the derivative of~$\mathscr{L}$ with respect to~$(u,y)$, that is
$$
\mathscr{L}'(u,y,q) = \begin{bmatrix} e_u(u,y)^\ast q + \mathcal{J}_u(u,y) \\ e_y(u,y)^\ast q + \mathcal{J}_y(u,y) \end{bmatrix}, \quad \text{with} \quad e_u = \frac{\partial}{\partial u}e(u,y), \quad e_y = \frac{\partial}{\partial y}e(u,y).
$$
The SQP method consists in a Newton method applied to the necessary optimality condition~\eqref{eq:FOCLagrange} in order to iteratively solve for the solution~$(u_\star,y_\star,q_\star)$ of the original problem~\eqref{eq:J}--\eqref{eq:sys}. Each Newton step results in the following linear system for the updates:
\begin{align*}
    \begin{bmatrix} \mathscr{L}''(u^{(k)},y^{(k)},q^{(k)}) & e'(u^{(k)},y^{(k)})^\ast\\ e'(u^{(k)},y^{(k)}) & 0\end{bmatrix} \begin{bmatrix} \begin{bmatrix} u^{(k+1)} - u^{(k)} \\ y^{(k+1)} - y^{(k)} \end{bmatrix} \\ q^{(k+1)} - q^{(k)} \end{bmatrix} = - \begin{bmatrix} e_u(u^{(k)},y^{(k)})^\ast q^{(k)} + \mathcal{J}_u(u^{(k)},y^{(k)}) \\ e_y(u^{(k)},y^{(k)})^\ast q^{(k)} + \mathcal{J}_y(u^{(k)},y^{(k)}) \\ e(u^{(k)},y^{(k)}) \end{bmatrix}.
\end{align*}
A sufficient second-order condition for optimality of~$(u_\star,y_\star,q_\star)$ is
\begin{align}\label{eq:suffLagrangian}
    \mathscr{L}''(u_\star,y_\star,q_\star)((v,x),(v,x)) \ge \kappa \|(v,x)\|_{U_T\times W_T(L^2(\mathfrak{S};V),L^2(\mathfrak{S};V'))}^2,
\end{align}
for~$\kappa >0$, and for all~$(v,x)$ in the kernel of~$e'(u_\star,y_\star)$, which is given by~$\{(v,x) \in U_T \times W_T(L^2(\mathfrak{S};V),L^2(\mathfrak{S};V'))) \mid \dot{x} = \mathcal{A}x + Bv, \, x(0) = 0\}$. By the linearity of the constraint~$e(u,y)$ we have~$e_{uu} = e_{uy} = e_{yu} = e_{yy} = 0$, and thus~$\mathscr{L}''(u,y,q) = \mathcal{J}''(u,y)$. Further, from Section~\eqref{sec:risk-adj} we know that~$\mathcal{J}''(u,y)((v,x),(v,x)) \ge \|v\|^2_{U_T}$. With the bound~$\|x\|_{W_T(L^2(\mathfrak{S};V),L^2(\mathfrak{S};V'))} \le C \|u\|_{U_T}$ for all~$(v,x)$ in the kernel of~$e'(u_\star,y_\star)$, we further estimate~$\mathcal{J}''(u,y)((v,x),(v,x)) \ge \frac{1}{2}\|v\|^2_{U_T} + \frac{1}{2C^2}\|x\|^2_{W_T(L^2(\mathfrak{S};V),L^2(\mathfrak{S};V'))}$. Thus, the sufficient condition~\eqref{eq:suffLagrangian} is satisfied for~$\kappa = \min\{\frac{1}{2},\frac{1}{2C^2}\}$. Moreover, the Newton step simplifies to
\begin{align*}
    \begin{bmatrix} \mathbf{1}_U & 0 &e_u(u^{(k)},y^{(k)})^\ast \\ 0 & \mathcal{J}_{yy}(u^{(k)},y^{(k)}) & e_y(u^{(k)},y^{(k)})^\ast \\ e_u(u^{(k)},y^{(k)}) & e_y(u^{(k)},y^{(k)}) & 0\end{bmatrix} \begin{bmatrix} \begin{bmatrix} u^{(k+1)} - u^{(k)} \\ y^{(k+1)} - y^{(k)} \end{bmatrix} \\ q^{(k+1)}  \end{bmatrix}  = - \begin{bmatrix}  \mathcal{J}_u(u^{(k)},y^{(k)}) \\  \mathcal{J}_y(u^{(k)},y^{(k)}) \\ e(u^{(k)},y^{(k)}) \end{bmatrix},
\end{align*}
which is the optimality system of~\eqref{eq:quadJ} subject to~\eqref{eq:sys}. Consequently, in each Newton step one can equivalently solve the quadratic subproblem~\eqref{eq:quadJ} (or~\eqref{eq:LQR}) subject to~\eqref{eq:sys} with expansion trajectory~$\bar{y} = y^{(k)}$.

In order to ensure that the Fr\'echet derivatives are well-defined at all iterates, we recall from Section~\ref{sec:section3} that this is ensured for pairs~$(u,y) \in Z$, i.e., pairs with additional state regularity~$y \in \mathcal{C}([0,T];\mathcal{C}(\mathfrak{S};H))$. Thus, we choose an initial trajectory~$\bar{y}^{(-1)} \in \mathcal{C}([0,T];\mathcal{C}(\mathfrak{S};H))$. From the Newton step above it is evident that~$e(u_{k+1},y_{k+1}) = 0$, and thus~$(u_{k+1},y_{k+1}) \in Z$, i.e., the iterates remain in~$Z$. We next show that~$Z$ is a closed (affine) subspace of~$X := L^2(0,T;U) \times W_T(L^2(\mathfrak{S};V),L^2(\mathfrak{S};V'))$: clearly~$Z \subset X$. To verify that~$Z$ is an affine subspace, let
$$
Z_0 := \{ (u,y) \in L^2(0,T;U) \times W_T(L^2(\mathfrak{S};V),L^2(\mathfrak{S};V')) \mid \dot{y} - \mathcal{A} y - Bu= 0,\,y(0) = 0\}.
$$
It is readily verified that~$Z_0$ is a linear subspace of~$X$. Furthermore, for arbitrary~$f \in L^2(0,T;V')$ and~$y_0 \in H$, there exists a pair~$(v_0,x_0)$ such that~$\dot{x}_0 - \mathcal{A} x_0 - Bv_0 -f= 0$ with~$x_0(0) = y_0$. Furthermore, any pair~$(v,x)\in Z$ can be expressed as a translation~$(v_0,x_0) + (\tilde{v}_0,\tilde{x}_0)$, where~$(\tilde{v}_0,\tilde{x}_0) \in Z_0$, since~$\dot{x} - \dot{x}_0 - \mathcal{A} (x - x_0) - B(v-v_0) = 0$, and~$x(0) - x_0(0) = 0$. Thus,~$Z = (v_0,x_0) + Z_0$ is an affine subspace of~$X$. We will next show that~$Z$ is closed in~$X$. For~$n \in \bbN$ let~$(u_n,y_n) \in Z$ be a sequence such that~$(u_n,y_n) \to (u,y) \in U_T \times W_T(L^2(\mathfrak{S};V),L^2(\mathfrak{S};V'))$. Then, it holds that~$Bu_n \to Bu$ since~$B$ is a bounded linear operator,~$\dot{y}_n \to \dot{y}$, and~$\mathcal{A}y_n \to \mathcal{A}y$ by continuity of~$\mathcal{A}$. Thus, taking the limit of~$\dot{y}_n(\bssigma) = \mathcal{A} y_n(\bssigma) + B u_n + f$,~$y_n(\bssigma;0,\cdot) = y_0$, we arrive at
\begin{align*}
    \dot{\bar{y}}(t) = \mathcal{A} \bar{y}(t) + B \bar{u}(t) + f(t), \quad\text{for a.e.~}t\in [0,T],\quad\text{and}\quad \bar{y}(\cdot;0,\cdot) = y_0,
\end{align*}
that is,~$Z$ is closed. Further,~$\mathcal{J}$ and~$e$ are smooth with locally Lipschitz continuous second derivatives, and~$e'(u_\star,y_\star)$ is surjective. Under these conditions it is known (see, \cite[Ch.~5.3]{ItoKunisch}) that the SQP method is locally quadratic convergent to the unique global minimizer of~\eqref{eq:J}--\eqref{eq:sys}, that is
\begin{align*}
    \|(u^{(k+1)},y^{(k+1)},q^{(k+1)}) - (u_\star,y_\star,q_\star)\|_{W} \le K \|(u^{(k)},y^{(k)},q^{(k)}) - (u_\star,y_\star,q_\star)\|_{W}^2,
\end{align*}
for a constant~$K$ independent of~$k$, provided that~$\|(u^{(0)},y^{(0)},q^{(0)}) - (u_\star,y_\star,q_\star)\|_{W}$ with~$W = U_T \times W_T(L^2(\mathfrak{S};V),L^2(\mathfrak{S};V'))) \times L^2(0,T;L^2(\mathfrak{S};V))\times L^2(\mathfrak{S};H)$ is sufficiently small.

\begin{algorithm}
\label{alg:SQP}
\begin{algorithmic}[1]SQP Algorithm
\Require Tolerance~$\varepsilon>0$. Initial expansion trajectory $\bar{y}^{(-1)} \in \mathcal{C}([0,T];\mathcal{C}(\mathfrak{S};H))$.
\Ensure Approximate minimizer~$(u_\star,y_\star)$ of~\eqref{eq:J}--\eqref{eq:sys} with control in feedback form.
\State Set $k \gets 0$ and~$y^{(k-1)}_\star \gets \bar{y}^{(-1)}$.
\Repeat
\State Given the expansion point~$\bar{y} = y^{(k-1)}_\star$, solve the equations in Theorem~\ref{thm:main} to find 
\NoLNState the optimal Riccati-based feedback control~$u_\star^{(k)}$ with corresponding closed-loop 
\NoLNState state~$y_\star^{(k)}$ of~\eqref{eq:LQR} s.t.~\eqref{eq:sys}. 
    \State $k \gets k + 1$.
\Until{$\|\nabla J(u_\star^{(k)})\|_{U_T} < \varepsilon$}.
\State \Return $(u^{(k)}_\star,y^{(k)}_\star) \approx (u_\star,y_\star) $, where the control is given in feedback form.
\end{algorithmic}
\end{algorithm}

The optimal control that is computed using Algorithm~\ref{alg:SQP} is in feedback form for the linear quadratic approximation of the original problem~\eqref{eq:J} subject to~\eqref{eq:sys} evaluated at the optimal trajectory.

\section{Discretization and numerical implementation}

This section is devoted to the numerical implementation of the dynamical systems appearing in Theorem~\ref{thm:main}.

\subsection{Numerical integration}
The expected values involved in the presented problem are given as integrals over an infinite-dimensional parameter space~$\mathfrak{S}$, see~\eqref{eq:exp_int}. For the numerical implementation we first truncate the infinite-dimensional parameter sequence~$\bssigma$ to an~$s$-dimensional vector~$\bssigma_s$ by setting all components with index larger than~$s$ equal to zero. In the follwoing we will thus consider~$\bssigma \approx (\bssigma_s,\boldsymbol{0}) = (\sigma_1,\sigma_2,\ldots,\sigma_s,0,0,\ldots)$. This allows to replace the infinite-dimensional integrals by integrals over the~$s$-dimensional parameter space~$\mathfrak{S} := [-1,1]^s$. 

The numerical approximation of these integrals using full tensorgrids of univariate quadrature rules suffers from the curse of dimensionality, i.e., the number of nodes required to guarantee a prescribed error tolerance increases exponentially in the dimension~$s$. For sufficiently smooth problems this issue can be alleviated using sparse grids~\cite{Bungartz_Griebel_2004}, or quasi-Monte Carlo methods~\cite{Dick_Kuo_Sloan_2013}. In this work we shall focus on sample average approximations of the respective integrals. This includes Monte Carlo as well as quasi-Monte Carlo methods and the presented results can straightforwardly be modified to other quadrature methods.

\subsection{Polynomial chaos expansion}
Let~$L_n(\xi)$ denote the Legendre
polynomial of degree~$n \ge 0$ in~$[-1,1]$, which is normalized such that
$$
\int_{-1}^{1} |L_n(\xi)|^2 \frac{\mathrm d\xi}{2} = 1
$$
Then we have~$L_0 = 1$ and~$\{L_n\}_{n\ge 0}$ is an orthonormal basis of~$L^2(-1, 1)$. For~$\bsnu \in \mathscr{F}$ we introduce the tensorized Legendre polynomials
$$ L_{\bsnu}(\bssigma) = \prod_{j \ge 1} L_{\nu_j} (\sigma_j).
$$
The tensorized Legendre polynomials~$\{L_\bsnu(\bssigma) \mid \bsnu \in \mathscr{F}\}$ form a Riesz basis, i.e. a dense, orthonormal
family in~$L^2(\mathfrak{S}; \mathrm d\bssigma)$; in particular, each~$\delta \in L^2(\mathfrak{S}; \mathrm d\bssigma; H)$ admits an orthogonal expansion
$$
\delta(\bssigma) = \sum_{\bsnu \in \mathscr{F}} \delta_\bsnu L_\bsnu(\bssigma), \quad \text{where} \quad \delta_\bsnu = \int_{\mathfrak{S}} \delta(\bssigma) L_\bsnu(\bssigma) \,\mathrm d\bssigma\in H
$$

In this manuscript we shall be interested in functions in~$\delta \in L^2(0,T;L^2(\mathfrak{S};\mathrm d\bssigma;H))$, thus for almost every~$t \in [0,T]$ the corresponding coefficients~$\{\delta_\bsnu\}_{\bsnu \in \mathscr{F}}$ are elements in~$H$.

\subsubsection{Stochastic Galerkin approximation of a parametric dynamical system}

We wish to approximate the solution~$y(\boldsymbol{\sigma}; t)$ of~\eqref{eq:sys} using a (truncated) generalized polynomial chaos expansion
\begin{equation}\label{eq:stochGapprox}
y(\boldsymbol{\sigma}; t) \approx \sum_{\boldsymbol{\nu} \in \Lambda} y_{\boldsymbol{\nu}}(t) L_{\boldsymbol{\nu}}(\boldsymbol{\sigma}),
\end{equation}
where~$\Lambda := \{ \bsnu \in \mathscr{F} : |\bsnu|\le p\}$ is a finite index set, and~$ L_{\boldsymbol{\nu}}(\boldsymbol{\sigma})$ are the tensorized orthonormal Legendre polynomials. The bound~$p$ on the cardinality of the multi-indices in~$\Lambda$ equals the highest degree of the polynomials, and hence the total number of coefficients in a truncated expansion of the form~\eqref{eq:stochGapprox} is~$K+1 = \frac{(s+p)!}{s!p!}$. Thus, we can number the indices as~$\Lambda = \{\bsnu_1,\bsnu_2,\ldots,\bsnu_{K+1}\}$. 
Furthermore, for almost every~$t \in [0,T]$ the coefficients~$\{y_\bsnu(t)\}_{\bsnu \in \mathscr{F}}$ are elements in~$H$.

\paragraph{General Case}

Substituting the polynomial expansion into the equation \eqref{eq:sys} yields
\begin{align*}
\sum_{\bsnu \in \Lambda} \dot{y}_{\bsnu}(t) L_{\bsnu}(\boldsymbol{\sigma}) = {A}(\boldsymbol{\sigma}) \sum_{\bsnu \in \Lambda} y_{\bsnu}(t) L_{\bsnu}(\boldsymbol{\sigma}) + Bu(t)+ f(t).
\end{align*}
Next we project both sides onto the basis functions~$\{L_{\bsm}\}_{\bsm \in \Lambda}$ by taking the inner product in~$L^2(\mathfrak{S};\mathrm d\bssigma)$
\begin{align*}
\Big\langle \sum_{\bsnu \in \Lambda} \dot{y}_{\bsnu}(t) L_{\bsnu}(\boldsymbol{\sigma}), L_{\bsm}(\boldsymbol{\sigma}) \Big\rangle = 
\Big\langle {A}(\boldsymbol{\sigma}) \sum_{\bsnu \in \Lambda} y_{\bsnu}(t) L_{\bsnu}(\boldsymbol{\sigma}), L_{\bsm}(\boldsymbol{\sigma}) \Big\rangle + \Big\langle B u(t) + f(t), L_{\bsm}(\boldsymbol{\sigma}) \Big\rangle,
\end{align*}
for~$\bsm \in \Lambda$. Using the orthonormality of the basis, this simplifies to 
\begin{align*}
\dot{y}_{\bsm}(t) = \sum_{\bsnu \in \Lambda} {A}^{(\bsnu, \bsm)} y_{\bsnu}(t) + \delta_{\boldsymbol{m}, \mathbf{0}}(Bu(t)+f(t)),\quad \bsm \in \Lambda,
\end{align*}
where the operator-valued coefficients are defined as
\begin{align*}
{A}^{(\bsnu, \bsm)} := \left\langle {A}(\boldsymbol{\sigma}) L_{\bsnu}(\boldsymbol{\sigma}), L_{\bsm}(\boldsymbol{\sigma}) \right\rangle_{L^2(\mathfrak{S};\mathrm d\bssigma)} \in \mathcal{L}(V,V'),
\end{align*}
and
\begin{align*}
\delta_{\boldsymbol{m}, \mathbf{0}}(Bu(t)+f(t)) = \left\langle B u(t) + f(t), L_{\bsm}(\boldsymbol{\sigma}) \right\rangle_{L^2(\mathfrak{S};\mathrm d\bssigma)} = \begin{cases}
Bu(t) + f(t), & \text{if } \bsm = \mathbf{0}, \\
0, & \text{otherwise}.
\end{cases},
\end{align*}
since~$B$,~$u$, and~$f$ are deterministic.
Denoting~${A}^{\text{GPC}} := \left( {A}^{(\bsnu, \bsm)} \right)_{\bsm, \bsnu \in \Lambda}$, the Galerkin system becomes the fully coupled system
\begin{align*}
\dot{\mathbf{y}}(t) = {A}^{\text{GPC}} \mathbf{y}(t) + (B u(t) + f(t)) \otimes e_0,
\end{align*}
where~$\mathbf{y}(t) = [y_{\bsnu_1}(t), \dots, y_{\bsnu_{K+1}}(t)]^\top \in H^{K+1} = \times_{j = 0}^K H$ is the stacked coefficient vector, and $(B u(t) + f(t)) \otimes e_0 = [Bu(t) +f(t), 0, \dots, 0]^\top$ is the projection onto the first polynomial.

\paragraph{Affine parameter dependence}

In many situations, such as in diffusion or reaction problems in which the random input field is parameterized in terms of a Karhunen--Lo\`eve expansion, the operator~${A}(\boldsymbol{\sigma})$ depends affinely on the parameters, that is
\begin{align*}
{A}(\boldsymbol{\sigma}) = {A}_0 + \sum_{j=1}^{s} \sigma_j {A}_j,
\end{align*}
where~${A}_j \in \mathcal{L}(V,V')$ for~$j = 0, \dots, s$, for some~$s \in \mathbb{N}$. Substituting this into the Galerkin projection and using orthonormality, we get
\begin{align*}
\dot{y}_{\boldsymbol{m}}(t) = {A}_0 y_{\boldsymbol{\nu}}(t) + \sum_{j=1}^{s} \sum_{\boldsymbol{\nu} \in \Lambda} {A}_j y_{\boldsymbol{\nu}}(t) \langle \sigma_j L_{\boldsymbol{\nu}}(\boldsymbol{\sigma}), L_{\boldsymbol{m}}(\boldsymbol{\sigma}) \rangle + \delta_{\boldsymbol{m}, \mathbf{0}} (B u(t) + f(t)), \quad \bsm \in \Lambda.
\end{align*}
Since the basis functions are tensorized polynomials, and we use a product measure, we can factor the ($s$-dimensional) integral over the parameter domain and use the orthonormality
\begin{align*}
[M_j]_{\boldsymbol{\nu}, \boldsymbol{\mu}} := \langle \sigma_j L_{\boldsymbol{\nu}}(\boldsymbol{\sigma}), L_{\boldsymbol{m}}(\boldsymbol{\sigma}) \rangle_{L^2(\mathfrak{S};\mathrm d\bssigma)} = \prod_{i=1, i \neq j}^s \delta_{\nu_i, m_i} \cdot \int_{-1}^1 \sigma_j L_{\nu_j}(\sigma_j) L_{m_j}(\sigma_j) \,\frac{\mathrm d\sigma_j}{2},
\end{align*}
and thus only univariate integrals need to be computed. Then, the full Galerkin system becomes
\begin{align*}
\dot{\mathbf{y}}(t) = \left( {A}_0 \otimes I + \sum_{j=1}^{s} {A}_j \otimes M_j \right) \mathbf{y}(t) + (B u(t) +f(t)) \otimes e_0,
\end{align*}
where~$\otimes$ denotes the tensor product.

\subsection{Choice of the control operator~$B$}
The control input~$u \in U_T$ is modeled as a deterministic function, that is, it is constant as a function of~$\bssigma \in \mathfrak{S}$. The polynomial chaos expansion of the control takes the form
$$
u(t) = u(\bssigma;t) = \sum_{\bsnu \in \mathscr{F}} u_{\bsnu}(t) L_{\bsnu}(\bssigma), \quad \text{where} \quad\begin{cases} u_\bsnu(t) = u(t) \in U_T & \text{if }\bsnu = \boldsymbol{0},\\
u_\bsnu(t) = 0 & \text{otherwise}.\end{cases}
$$
Thus, in the (truncated) stochastic Galerkin basis the control operator~$B: U \to H^{K+1}$ can be chosen to act only on the first polynomial, that is, the zeroth-order coefficients
$$\begin{bmatrix} B u_{\bsnu_1}(t), B u_{\bsnu_2}(t),\ldots, B u_{\bsnu_{K+1}}(t)\end{bmatrix} = \begin{bmatrix} B u_{\bsnu_1}(t), 0,\ldots, 0\end{bmatrix}= \begin{bmatrix} B u(t), 0,\ldots, 0\end{bmatrix},$$
where~$\bsnu_1 = \boldsymbol{0}$. Below, by slight abuse of notation we use the same symbol for the control operator~$B: U\to H$, its extension to the Bochner space~$B: U \to L^2(\mathfrak{S};H)$, and the Galerkin discretization of the latter~$B: U \to H^{K+1}$. 

We will now show, that the adjoint operator of~$B$ naturally involves the expected value. Indeed, for~$v \in L^2(\mathfrak{S};H)$, we have
\begin{align*}
    \langle u, B^\ast v\rangle_{U} = \langle B u,v\rangle_{L^2(\mathfrak{S};H)} &= \int_{\mathfrak{S}} \langle B u, v(\bssigma) \rangle_H \,\mathrm d\bssigma = \int_{\mathfrak{S}} \langle u, B^\ast v(\bssigma) \rangle_H \,\mathrm d\bssigma \\&= \langle u, \int_{\mathfrak{S}} B^\ast v(\bssigma) \,\mathrm d\bssigma  \rangle_U = \langle u, B^\ast \int_{\mathfrak{S}} v(\bssigma) \,\mathrm d\bssigma  \rangle_U.
\end{align*}

With this choice of the control operator, the optimal control of a linear quadratic problem in the polynomial chaos basis is then implemented as
\begin{align}\label{eq:ufbGPC}
\begin{split}
    u(t) = -B^\ast \Pi(T-t) y(\bssigma;t) &\approx - \sum_{\bsnu \in \Lambda} B^\ast \Pi_\bsnu(T-t)  y_\bsnu(t) L_\bsnu(\bssigma) \\ &=- B^\ast \bigg( \sum_{j = 1}^{K+1} (\Pi_{\bsnu_1,\bsnu_j}(T-t) y_{\bsnu_j}(t)\bigg)L_{\bsnu_1}(\bssigma),
\end{split}
\end{align}
where for~$\bsnu 
\in \Lambda$ and all~$t\in [0,T]$, we use the notation
$$
\Pi_{\bsnu}(T-t) = \begin{bmatrix}
    \Pi_{\bsnu_1,\bsnu_1} & \Pi_{\bsnu_1,\bsnu_2} & \cdots & \Pi_{\bsnu_{1},\bsnu_{K+1}}\\
    \Pi_{\bsnu_2,\bsnu_1} & \ddots & & \vdots\\
    \vdots  & & \ddots &\vdots\\
    \Pi_{\bsnu_{K+1},\bsnu_1} & \Pi_{\bsnu_{K+1},\bsnu_2} & \cdots & \Pi_{\bsnu_{K+1},\bsnu_{K+1}}
\end{bmatrix}(T-t) \in \mathcal{L}(H^{K+1},H^{K+1}).
$$
The polynomial expansion of the feedback control in~\eqref{eq:ufbGPC} contains only zeroth-order polynomials, if $\bsnu_1 = \boldsymbol{0}$. Thus, the optimal control is a deterministic function, that is, it is constant as a function of the random parameters~$\bssigma \in \mathfrak{S}$.

\paragraph{Related work} Suppose there is a true but unknown parameter~$\bssigma = \bar{\bssigma}\in \mathfrak{S}$. In this situation one can apply the following version of the above feedback: using the same~$B$ and~$\{\Pi_{\bsnu}\}_{\bsnu \in \Lambda}$, we project the state trajectory onto the first index~$\bsnu_1$ corresponding to the zeroth-order polynomials. For~$t\in [0,T]$ and a given finite set of multi-indices~$\Lambda \subset \mathscr{F}$, the resulting feedback takes the form~$\{-B^\ast \Pi_{\bsnu}(T-t) \mathcal{E}\}_{\bsnu \in \Lambda} = -B^\ast \Pi_{\bsnu_1,\bsnu_1}(T-t)$, where~$\mathcal{E}: H \to H^{K+1}$ maps~$y \mapsto \begin{bmatrix} y^\top & 0& \cdots&0\end{bmatrix}^\top$. This feedback law has been constructed and investigated in~\cite{GuthKunRod23, GuthKunRod24} based on snapshots/samples, whereas in this manuscript we use a stochastic Galerkin ansatz.

\subsection{Choice of the expansion point~$\bar{y}$ and approximation of the weights}
The expansion point~$\bar{y}$, around which the entropic risk objective functional is approximated, is updated according to Algorithm~\ref{alg:SQP}. The initialization of Algorithm~\ref{alg:SQP} can be constructed as follows. In practical applications, when observational data are available,~$\bar{y}^{(-1)}$ may be constructed from filtered or regularized measurements, yielding a data-informed reference that incorporates empirical knowledge of the state. In the absence of data one might start, for instance, with the uncontrolled trajectories.

In every iteration of Algorithm~\ref{alg:SQP}, the weight functions are updated by evaluation at the current expansion point~$\bar{y}$, which is given as a polynomial chaos surrogate (cp.~\eqref{eq:stochGapprox}), that is
\begin{align*}
    \bar{y}(\bssigma;t) = \sum_{\bsnu \in \mathscr{F}} \bar{y}_\bsnu(t) L_\bsnu(\bssigma) \approx \sum_{\bsnu \in \Lambda} \bar{y}_\bsnu(t) L_\bsnu(\bssigma),
\end{align*}
for a finite set of multi-indices~$\Lambda \subset \mathscr{F}$. 

We select a set of quadrature nodes~$\{\bssigma^{(1)},\ldots,\bssigma^{(N)}\}$, leading, for~$t\in [0,T]$, to an ensemble of expansion trajectories~$\{\bar{y}(\bssigma^{(1)};t),\ldots,\bar{y}(\bssigma^{(N)};t)\}$, and to a collection of weights stored in a vector~$\boldsymbol{\omega}$ of length~$N$ with components 
\begin{align}\label{eq:weights}
    \omega_i(t) := \omega_{\theta,C}(\bar{y}(\bssigma^{(i)};t)) = \frac{\mathrm e^{\theta \|C(\bar{y}(\bssigma^{(i)};t)-g(t))\|_H^2}}{\frac{1}{\tilde{N}} \sum_{k=1}^{\tilde N} \left[\mathrm e^{\theta \|C(\bar{y}(\bssigma^{(k)};t)-g(t))\|_H^2} \right]},
\end{align}
with a possibly different set of quadrature nodes~$\{\bssigma^{(1)},\ldots,\bssigma^{(\tilde{N})}\}$,~$\tilde{N} \in \bbN$. We point out that each node is an~$s$-dimensional vector, that is~$\bssigma^{(i)} \in [-1,1]^s$ for all~$1 \le i \le N$. The weights~$\omega_{\theta,P}(\bar{y}(\bssigma^{(i)};T))$ are computed analogously.

\subsection{Approximation of the weighted covariance}
Given~$t \in [0,T]$ and a polynomial chaos expansion of the test functions~$\delta^X(t)\in L^2(\mathfrak{S};H)$ and~$\delta^Y(t)\in L^2(\mathfrak{S};H)$, the weighted sample average approximation of the covariance term in the bilinear form associated to the operator~$\mathcal{Q}_{C,\omega}(\bar{y}(\cdot;t);t)$ is of the form
\begin{align*}&{\rm Cov}_{\omega_{\theta,C}(y(\bssigma;t))}\Big( \langle \overline{\mathscr{C}}(\cdot;t), \delta^X(\cdot;t) \rangle_H, \langle \overline{\mathscr{C}}(\cdot;t), \delta^Y(\cdot;t) \rangle_H\Big) \\
&\quad\approx\frac{1}{N-1}\sum_{i=1}^N \bigg( \omega_i(t) \bigg(X_i(t) - \frac{1}{N}\sum_{j=1}^N \omega_j(t) X_j(t)\bigg) \bigg(Y_i(t) - \frac{1}{N}\sum_{j=1}^N \omega_j(t) Y_j(t)\bigg)\bigg),\end{align*} where, for~$t \in [0,T]$,~$X(t)$ and~$Y(t)$ are two vectors of length~$N$, which contain
\begin{align*}
    X_i(t) &:= \sum_{\bsm,\bsnu \in \mathscr{F}} \langle (C^\ast C  \bar{y}_\bsnu(t) - \delta_{\bsnu,\boldsymbol{0}}g(t)) L_\bsnu(\bssigma^{(i)}), \delta^X_\bsm(t)L_\bsm(\bssigma^{(i)})\rangle_H \\
    Y_i(t) &:= \sum_{\bsm,\bsnu \in \mathscr{F}} \langle (C^\ast C  \bar{y}_\bsnu(t) - \delta_{\bsnu,\boldsymbol{0}}g(t)) L_\bsnu(\bssigma^{(i)}), \delta^Y_\bsm(t) L_\bsm(\bssigma^{(i)})\rangle_H 
\end{align*}
for~$1\le i \le N$. In the remainder of this section, we shall omit the dependence on~$t$ for better readability whenever it is clear from the context. The (unbiased) weighted empirical covariance for fixed~$t \in [0,T]$ admits a matrix representation, that is
\begin{align*}    \langle \mathfrak{Cov}(\omega) X,Y\rangle_{\bbR^N} =  \frac{1}{N-1}\sum_{i=1}^N \bigg( \omega_i \bigg(X_i - \frac{1}{N}\sum_{j=1}^N \omega_j X_j\bigg) \bigg(Y_i - \frac{1}{N}\sum_{j=1}^N \omega_j Y_j\bigg)\bigg).
\end{align*}
It is readily verified that this matrix representation is given by
\begin{align*}
    \mathfrak{Cov}(\omega) := \frac{1}{N(N-1)}\begin{bmatrix}
        c_{ij}
    \end{bmatrix} \quad \text{with} \quad c_{ij} := \begin{cases}
        N\omega_i - \omega_i^2 & \text{if } i=j,\\
        -\omega_i \omega_j &\text{otherwise}.
    \end{cases}
\end{align*}

Using this matrix representation, we will next derive a matrix representation of the operator associated with the covariance term in~$\mathcal{Q}_{C,\omega}(\bar{y}(\cdot;t);t)$.

\paragraph{Approximation of the operator~$\mathcal{Q}_{C,\omega}(\bar{y}(\cdot;t);t)$}

For the numerical implementation, a spatial discretization, for instance by a finite element method (FEM), is required. For functions~$v \in H$ we suppose a representation of the form~$v = \sum_{i = 1}^{d_{\rm FEM}} v_i \phi_i$ such that~$\langle v, \tilde{v} \rangle_H = \bsv^\top M \tilde{\bsv}$, for~$v\in H$ and $\tilde{v} \in H$ with~$M_{ij} = \langle \phi_i, \phi_j\rangle_H$, i.e.,~$M \in \mathbb{R}^{d_{\rm FEM}\times d_{\rm FEM}}$. 

Without loss of generality we assume that~$\bsnu$ and~$\bsm$ both belong to~$\Lambda$, and that~$b_\bsm^X$ and~$b_\bsm^Y$ depend on the same multi-index~$\bsm$. Then, defining the FEM coefficient vectors as $\boldsymbol{a}_{\bsnu_k} := \begin{bmatrix}C^\ast C  \bar{y}_{\bsnu_k}(t) - \delta_{\bsnu_k,\boldsymbol{0}}g(t)\end{bmatrix} \in \mathbb{R}^{d_{\rm FEM}}$ and~$b^X_{\bsm_k}:= \begin{bmatrix} \delta^X_{\bsm_k}(t) \end{bmatrix} \in \mathbb{R}^{d_{\rm FEM}}$ for~$1\le k \le K+1$ allows to rewrite 
\begin{align*}
    X_i = \begin{bmatrix} a_{\bsnu_1}^\top & \ldots & a_{\bsnu_{K+1}}^\top \end{bmatrix} \boldsymbol{L}_\bsnu(\bssigma^{(i)}) \boldsymbol{M} \boldsymbol{L}_\bsm(\bssigma^{(i)})\begin{bmatrix} b^X_{\bsm_1} \\ \vdots \\ b^X_{\bsm_{K+1}}\end{bmatrix},
\end{align*}
where~$\boldsymbol{M} = \operatorname{blockdiag}(M) \in  \mathbb{R}^{d_{\rm FEM} \cdot (K+1) \times d_{\rm FEM} \cdot (K+1)}$ and
\begin{align*}
    \boldsymbol{L}_\bsnu(\bssigma^{(i)}) = \begin{bmatrix} [L_{\bsnu_1}(\bssigma^{(i)})] & & \\ & \ddots & \\ & & [L_{\bsnu_{K+1}}(\bssigma^{(i)})] \end{bmatrix} \in \mathbb{R}^{d_{\rm FEM} \cdot (K+1) \times d_{\rm FEM} \cdot (K+1)},
\end{align*}
with
\begin{align*}[L_{\bsnu_k}(\bssigma^{(i)})] \in \mathbb{R}^{d_{\rm FEM} \times d_{\rm FEM}} \quad \text{with entries}\quad [L_{\bsnu_k}(\bssigma^{(i)})]_{j\ell} = \begin{cases} L_{\bsnu_k}(\bssigma^{(i)}) & \text{if }j = \ell,\\ 0 & \text{otherwise}.\end{cases}\end{align*}
Analogously, with~$b^Y_{\bsm_k}:= \begin{bmatrix} \delta^Y_{\bsm_k}(t) \end{bmatrix} \in \mathbb{R}^{d_{\rm FEM}}$ for~$1\le k \le K+1$ we get
\begin{align*}
    Y_i = \begin{bmatrix} a_{\bsnu_1}^\top & \ldots & a_{\bsnu_{K+1}}^\top \end{bmatrix} \boldsymbol{L}_\bsnu(\bssigma^{(i)}) \boldsymbol{M}  \boldsymbol{L}_\bsm(\bssigma^{(i)})\begin{bmatrix} b^Y_{\bsm_1} \\ \vdots \\ b^Y_{\bsm_{K+1}}\end{bmatrix}.
\end{align*}
For the set of quadrature points~$\{\bssigma^{(1)},\ldots,\bssigma^{(N)}\}$ we define
\begin{align*}
    \mathfrak{M} = \begin{bmatrix}\begin{bmatrix}a_{\bsnu_1}^\top & \ldots & a_{\bsnu_{K+1}}^\top \end{bmatrix} \boldsymbol{L}_\bsnu(\bssigma^{(1)}) \boldsymbol{M}  \boldsymbol{L}_\bsnu(\bssigma^{(1)}) \\
    \vdots\\\begin{bmatrix}a_{\bsnu_1}^\top & \ldots & a_{\bsnu_{K+1}}^\top \end{bmatrix} \boldsymbol{L}_\bsm(\bssigma^{(N)}) \boldsymbol{M}  \boldsymbol{L}_\bsm(\bssigma^{(N)})\end{bmatrix} \in \mathbb{R}^{N \times d_{\rm FEM}\cdot (K+1)}.
\end{align*}
such that we can write
\begin{align*}  
X  = \begin{bmatrix}
        X_1 \\ \vdots \\ X_N
        \end{bmatrix} = \mathfrak{M} \begin{bmatrix} b^X_{\bsm_1} \\ \vdots \\ b^X_{\bsm_{K+1}}\end{bmatrix} \in \mathbb{R}^N\quad \text{and} \quad Y  = \begin{bmatrix}
        Y_1 \\ \vdots \\ Y_N
        \end{bmatrix} = \mathfrak{M} \begin{bmatrix} b^Y_{\bsm_1} \\ \vdots \\ b^Y_{\bsm_{K+1}}\end{bmatrix} \in \mathbb{R}^N.
\end{align*}
This allows the following representation of the discretized covariance operator
\begin{align*}
    \langle \mathfrak{Cov}(\omega) X,Y\rangle_{\bbR^N} &= \begin{bmatrix} X_1 \\ \vdots \\ X_N \end{bmatrix}^\top \mathfrak{Cov}(\omega) \begin{bmatrix} Y_1 \\ \vdots \\ Y_N \end{bmatrix} = 
    \begin{bmatrix} b^X_{\bsm_1} \\ \vdots \\ b^X_{\bsm_{K+1}}\end{bmatrix}^\top \mathfrak{M}^\top \mathfrak{Cov}(\omega) \mathfrak{M}\begin{bmatrix} b^X_{\bsm_1} \\ \vdots \\ b^X_{\bsm_{K+1}}\end{bmatrix}.
\end{align*}
where~$\mathfrak{M}^\top \mathfrak{Cov}(\omega) \mathfrak{M}$ is of dimension~$d_{\rm FEM}\cdot (K+1) \times d_{\rm FEM}\cdot (K+1)$.

Furthermore, for the constant term in~$\mathcal{Q}_{C,\omega}(\bar{y}(\cdot;t);t)$ we have 
\begin{align*}
\langle C^\ast C \delta^X, \delta^Y \rangle_{L^2(\mathfrak{S};H)} &= \sum_{\bsm, \bsnu \in \Lambda} \langle C^\ast C \delta^X_\bsnu, \delta^Y_\bsm \rangle_H \langle \omega_{\theta,C}(\bar{y}(\bssigma;t)) L_\bsnu(\bssigma), L_\bsm(\bssigma)\rangle_{L^2(\mathfrak{S};\mathrm d\bssigma)}\\
&= \begin{bmatrix} b^X_{\bsnu_1} \\ \vdots \\ b^X_{\bsnu_{K+1}}\end{bmatrix}^\top \begin{bmatrix} \frac{1}{N} \sum_{i=1}^N \omega_i L_{\bsnu}(\bssigma^{(i)}) \boldsymbol{C^\ast M C} L_{\bsm}(\bssigma^{(i)}) \end{bmatrix} \begin{bmatrix} b^Y_{\bsm_1} \\ \vdots \\ b^Y_{\bsm_{K+1}}\end{bmatrix},
\end{align*}
where~$\boldsymbol{C^\ast M C} = \operatorname{blockdiag}(C^\ast M C) \in \mathbb{R}^{d_{\rm FEM} \cdot (K+1) \times d_{\rm FEM} \cdot (K+1)}$. 

Overall, for~$\bsnu, \bsm \in \Lambda$,~$t \in [0,T]$ and~$\{\bssigma^{(i)}\}_{i = 1}^N$ we have
$$\mathcal{Q}_{C,\omega}(\bar{y}(\cdot;t);t) \approx \frac{1}{N} \sum_{i=1}^N \Big(\omega_i(t) L_{\bsnu}(\bssigma^{(i)}) \boldsymbol{C^\ast M C} L_{\bsm}(\bssigma^{(i)})\Big) + 2\theta\, \mathfrak{M}^\top(t) \mathfrak{Cov}(\omega;t) \mathfrak{M}(t).$$
The operator~$\mathcal{Q}_{P,\omega}$ can be computed analogously.

\subsection{Remarks and outlook on the error analysis}
In addition to the spatial and temporal discretization of a deterministic optimal control problem subject to PDEs, the involved uncertainty requires the following approximations: 
\begin{enumerate}
    \item Dimension truncation of the parameter space, leading to
    \begin{align}\label{eq:dimtrunc}
        \int_{\mathfrak{S}} X(\bssigma) \,\mathrm d\bssigma \approx \int_{\mathfrak{S}_s} X(\bssigma_s) \,\mathrm d\bssigma_s,
    \end{align}
    for some integrable quantity of interest~$X$.
    \item Numerical integration of the~$s$-dimensional integrals
    \begin{align}\label{eq:QMC}
        \int_{\mathfrak{S}_s} X(\bssigma_s) \,\mathrm d\bssigma_s \approx \sum_{i = 1}^N w_i X(\bssigma^{(i)}),
    \end{align}
    for some integrable quantity of interest~$X$.
    \item Truncation of the polynomial chaos expansion
    \begin{align}\label{eq:GPC}
        \sum_{\bsnu \in \mathscr{F}} \delta_\bsnu L_\bsnu(\bssigma) \approx \sum_{\bsnu \in \Lambda} \delta_\bsnu L_\bsnu(\bssigma),
    \end{align}
    for a finite set~$\Lambda \subset \mathscr{F}$ and some~$\delta \in L^2(\mathfrak{S}; \mathrm d\bssigma; H)$
\end{enumerate}
For sufficiently smooth problems these errors can be controlled. More precisely, assuming that~$\|\partial^\bsnu_\bssigma A(\bssigma)\|_{\mathcal{L}(V,V')} \le \bsb^{\bsnu}$ for a monotonically decreasing sequence~$\bsb = (b_j)_{j \in \bbN}$, which is~$p$-summable for some~$p \in (0,1)$, it can be shown that the solution of~\eqref{eq:sys} depends analytically on the parameter sequence~$\bssigma \in \mathfrak{S}$ with~$\|\partial^{\bsnu}_{\bssigma} y(\bssigma)\|_V \le C |\bsnu|! \bsb^{\bsnu}$.
Here we use the following notation: for a sequence~$\bssigma := (\sigma_j)_{j\in \bbN}$ of real numbers and~$\bsnu \in \mathscr{F}$, we define
\begin{align*}
        \partial^\bsnu_{\bssigma} := \frac{\partial^{\nu_1}}{\partial\sigma_1} \frac{\partial^{\nu_2}}{\partial\sigma_2} \cdots, \qquad \text{and} \qquad \bssigma^\bsnu := \prod_{j=1}^{\infty} \sigma_j^{\nu_j},
\end{align*}
where we follow the convention~$0^0 := 1$. 

In \cite{kunoth2013analytic} it is shown that the optimal control and optimal state-adjoint-state pair of a linear quadratic optimal control problem admit analytic regularity with respect to the uncertain parameters, provided that the system operator~$A(\bssigma)$ admits the regularity as described above. Based on this result, in~\cite{GKK24} it is shown that a Riccati based feedback law for an autonomous system ($A(\cdot),B,C,P$) admits a similar regularity bound. Such bounds are frequently obtained in the context of Karhunen--Lo\`eve expansions of random fields (\cite{becktemponenobiletamellini, Cohen2010, ChenQuarteroni, Graham2015}), and can be used to derive convergence rates for the approximations~\eqref{eq:dimtrunc}, \eqref{eq:QMC}, and~\eqref{eq:GPC}: precisely the integrated dimension truncation error~\eqref{eq:dimtrunc} decays as~$\mathcal{O}(s^{-2/p+1})$ (\cite{GuthKaa23}), while the~$L^2$-dimension truncation error decays as~$\mathcal{O}(s^{-1/p+1/2})$ (\cite{GuthKaa23_2}). Randomly shifted rank-1 lattice rules achieve a root mean square error of~$\mathcal{O}(N^{-1+\varepsilon})$ (\cite[Thm.~6.6]{GKKSS24}), while interlaced polynomial lattice rules achieve an integration error convergence~$\mathcal{O}(N^{-\alpha})$,~$\alpha\ge 1$ (\cite[Sect.~5]{GKK24}). Moreover, there exists a sequence of index sets~$(\Lambda_n)_{n \in \mathbb{N}}$ whose cardinality does not exceed~$n$ such that the best $n$-term approximation rate (in~$L^2$) of the Legendre polynomials is~$\mathcal{O}(n^{1/p+1/2})$, see~\cite{Cohen_DeVore_2015}, or~\cite[Thm.~7]{kunoth2013analytic}.

The extension of these results to the presented setting with a time-dependent term appearing in the Riccati equation as well as a parameter-dependent target of the quadratic subproblems, which are used to approximate the risk-averse optimal control problem, exceeds the scope of this paper and is postponed to future works.

\section{Numerical experiments}
Let us consider the parameterized diffusion-reaction equation under Neumann boundary conditions as follows
\begin{align*}
\dot{y}(\bssigma;t,x) - 0.5\Delta y(\bssigma;t,x) + c(\bssigma;x) y(\bssigma;t,x)  &= \sqrt{10}\sum_{i = 1}^{N_{\mathrm a}} u_i(t,x) \mathbf{1}_{O_i}(x) &&(t,x) \in (0,T] \times D,\\
\frac{\partial y(\bssigma;t,x)}{\partial \mathrm n} &= 0 &&(t,x) \in [0,T] \times \partial D,\\
y(\bssigma;t,x) &= y_\circ &&(t,x) \in \{t=0\} \times D,
\end{align*}
where~$T=0.5$,~$D = (0,1)$ with boundary~$\partial D = \{0,1\}$, and the functions~$\mathbf{1}_{O_i}$ represent the support of the actuators, which are modelled as the characteristic functions associated to open sets~$O_i \subset D$ for~$1\le i \le N_a$. In the following numerical experiments~$N_{\mathrm a} = 3$ actuators are used as~$O_1 = [0.1, 0.3]$, $O_2 = [0.4, 0.6]$, and $O_3 = [0.7, 0.9]$. While the diffusion is constant, the reaction term is modeled as 
\begin{align*}
    c(\bssigma;x) = \bar{c}(x) + \sum_{j=1}^{s} \sigma_j \psi_j(x),
\end{align*}
where the mean field is set to~$\bar{c}(x) = 0.2$ and~$\sigma_j$ are independent and identically distributed (i.i.d.)~uniformly in~$[-1,1]$ for all~$j=1,\ldots,s=2$. Furthermore, the parametric basis functions are chosen to be~$\psi_{2j}(x) = (2j)^{-\vartheta} \sin(j\pi x)$ and~$\psi_{2j-1}(x) = (2j-1)^{-\vartheta} \cos(j\pi x)$ with decay~$\vartheta = 2$. We use tensorized Legendre polynomials up to degree~$p = 2$ for the stochastic Galerkin method,~$(\Delta t)^{-1} = 200$ time steps for the temporal discretization and piecewise linear finite elements with meshwidth~$h = 2^{-5}$ for the spatial discretization. The weights~$\eqref{eq:weights}$ are approximated using~$100$ i.i.d.~samples of the trajectories in a Monte Carlo approximation for the expected values.

We set~$P=0$ and $Q = \mathbf{1}_H$, where~$\mathbf{1}_H$ denotes the identity operator in~$H$, as well as the initial condition~$y_\circ(x) = 4-\cos(2\pi x)$. The initial expansion point~$\bar{y}^{(-1)}$ for Algorithm~\ref{alg:SQP} solves the uncontrolled equation with shifted initial condition~$\bar{y}_\circ(t=0,x) = 1-\cos(2\pi x)$. The target~$g$ solves the state equation with~$c(\bssigma;x) = 0$ for all~$\bssigma \in \mathfrak{S}$ and all~$x \in D$, and with initial data~$g(t=0,x) = 1.25-\cos(2\pi x)$.

\paragraph{Comparison to the risk-neutral case}
In Figure~\ref{fig1}, we present the empirical distributions of the optimal tracking errors at six selected time points, based on~$10^4$ random draws of the parameter~$\bssigma \in \mathfrak{S}$. We compare the risk-averse ($\theta = 10$) and risk-neutral ($\theta = 0$) cases. As expected, the distribution of the optimal risk-averse tracking errors are shifted toward smaller errors compared to the optimal risk-neutral case. This behavior is observed uniformly across all selected time points.

\begin{figure}[ht]
\centering
{\includegraphics[trim={0 0 0 0},clip,width=1.\linewidth]{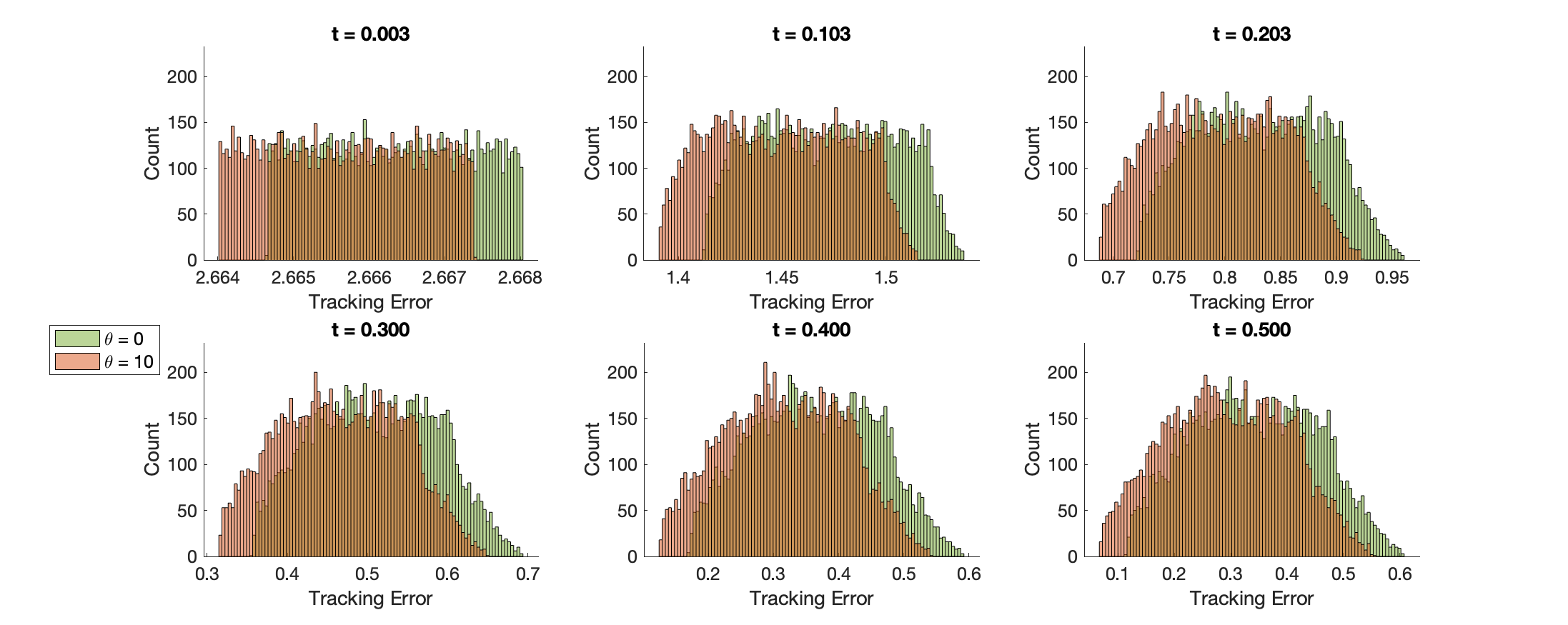}}
\caption{Histogram of the tracking errors at six selected time points based on~$10^4$ realizations of the optimal trajectory. Risk-averse case (green) and risk-neutral-case (orange).}
\label{fig1}
\end{figure}

\paragraph{Comparison to reference trajectories} As before we compute~$10^4$ realizations of the tracking errors of the uncontrolled state. At~$t = 0$ the tracking errors coincide for all realizations since the initial condition is deterministic. Furthermore, we observe that the tracking errors increase over time for all realizations, see Figure~\ref{fig:2}.
Figure \ref{fig:3} shows~$10^4$ realizations of the tracking error corresponding to the optimal state, obtained with the proposed SQP method (Algorithm \ref{alg:SQP}) alongside~$10^4$ realizations of the tracking error corresponding to an open-loop reference solution computed by a gradient descent algorithm. The displayed results corresponds to the final iterate after~$20$ SQP iterations and~$200$ gradient descent iterations, respectively. The SQP tracking errors are indistinguishable from the open-loop reference, indicating that the two methods converge approximately to the same solution, thus providing a validation for the numerical realizations.

\begin{figure}[ht]
     \centering
     \begin{subfigure}{0.49\textwidth}
         \centering
         {\includegraphics[trim={0 0 0 0},clip,width=1.\linewidth]{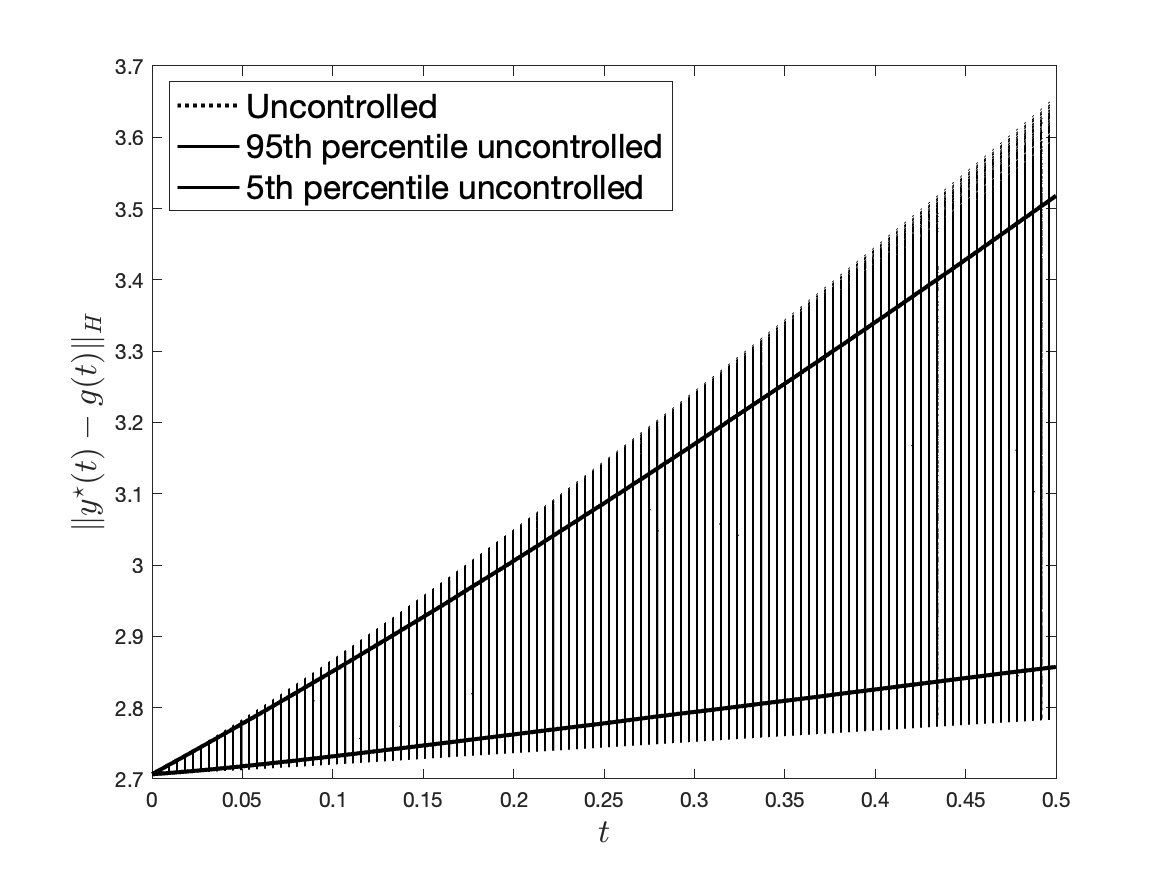}}
         \caption{$10^4$ realizations of the tracking error as a function of time for the uncontrolled state trajectories alongside the 95th and 5th percentiles.}
         \label{fig:2}
     \end{subfigure}
     \hfill
     \begin{subfigure}{0.49\textwidth}
         \centering
         {\includegraphics[trim={0 0 0 0},clip,width=1.\linewidth]{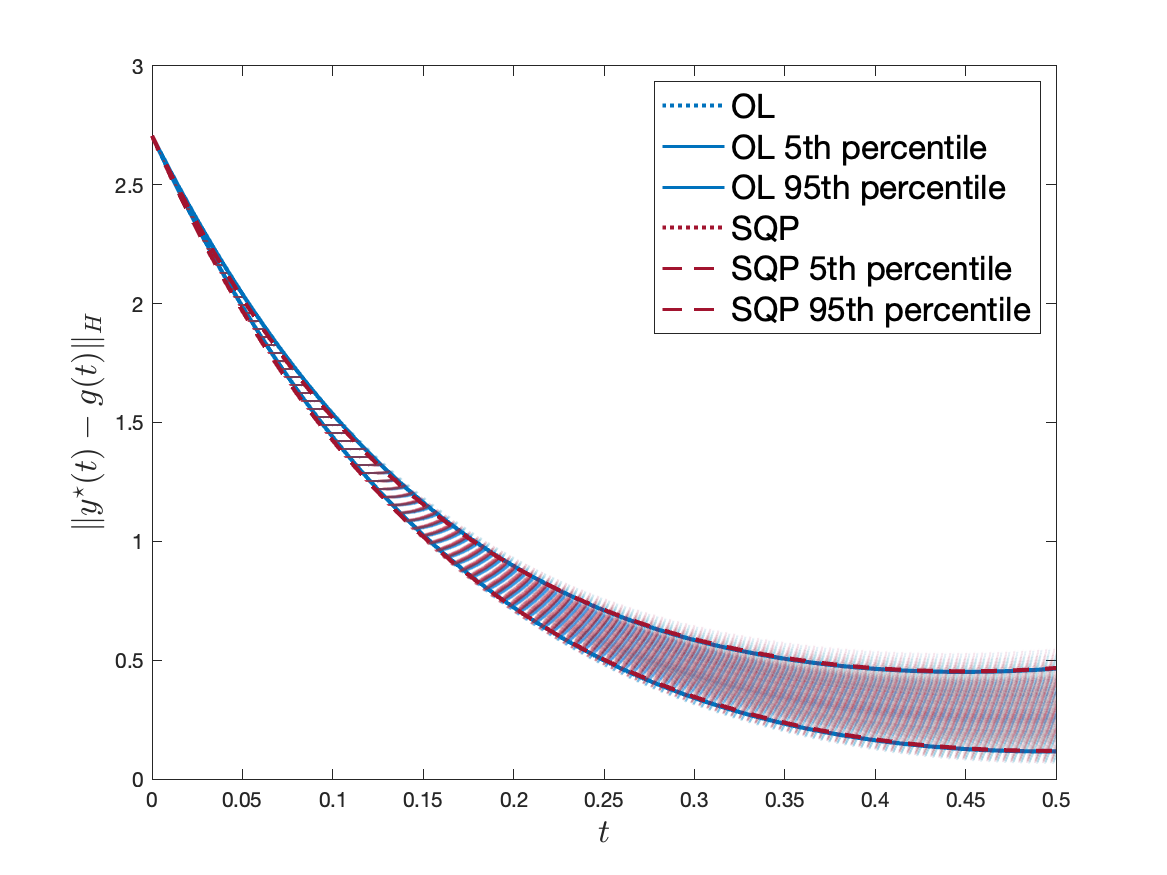}}
         \caption{$10^4$ realizations of the tracking error as a function of time for the optimal trajectories alongside the 95th and 5th percentiles (SQP and OL, respectively).}
         \label{fig:3}
     \end{subfigure}
\end{figure}

\paragraph{Robustness against perturbed initial conditions}
Closed-loop solutions are expected to behave in a more robust manner against noise than open-loop controls. In this example we illustrate that this is also the case in the present setting. To demonstrate this behavior, the performance of the approximate feedback control obtained by the proposed SQP method is compared to an open-loop solution for varying initial conditions. To this end, we parameterize the initial condition by a noise level~$\ell$ as
\begin{align*}
y_{\ell}^+ &:= y_0 + \ell + 0.01\ell \xi,\quad 0\le \ell \le 2,\\
y_{\ell}^- &:= y_0 - \ell - 0.01\ell \xi,\quad 0\le \ell \le 2.
\end{align*}
where~$\xi$ is a vector of the same length as~$y_0$ containing i.i.d.~standard normally distributed random variables. The approximate feedback control as well as the open-loop solution are computed without noise, i.e., for~$\ell = 0$ such that~$y_0 = y_\ell^+ = y_\ell^-$, and then tested in the dynamics with perturbed initial conditions. The results for~$y_\ell^+$ and~$y_\ell^-$ are displayed in Figure~\ref{fig4}. It is observed that the approximate closed-loop control (CL) steers the state to approximately the same neighborhood of the target, i.e., the tracking errors have a similar magnitude across different noise levels. The open-loop reference solution on the other hand clearly leads to larger tracking errors for increasing noise levels. Furthermore, with noise level~$\ell = 0$ the CL and OL solutions are indistinguishable as expected. 

\begin{figure}[h!]
\centering
{\includegraphics[width=1.\linewidth]{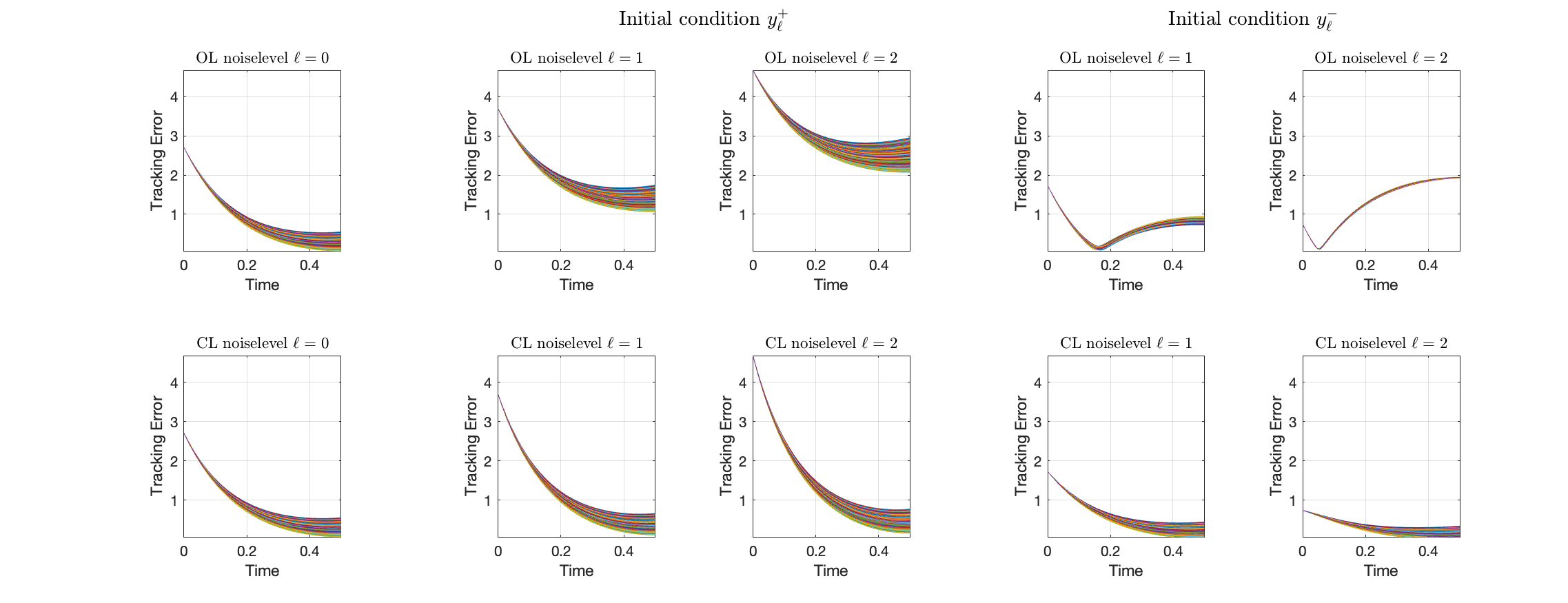}}
\caption{$10^4$ realizations of the tracking error for the open-loop solution (top line) and the approximate closed-loop solution (bottom line) for different initial conditions. The plots in the left column correspond to the unperturbed initial condition~$y_0$, i.e., to noise level~$\ell = 0$. The plots in the center correspond to the perturbed intial condition~$y_\ell^+$, $0\le \ell \le 2$, and the plots in the right-hand side block correspond to the perturbed intial condition~$y_\ell^-$, $0\le \ell \le 2$.}
\label{fig4}
\end{figure}

\section{Conclusions and outlook}

We investigate a class of risk-averse optimal control problems for linear dynamical systems with uncertain evolution operators and develop an SQP algorithm to approximate the optimal feedback control. Such controls in feedback form can be constructed independently of the initial condition of the dynamics, which makes them particularly well-suited for time-critical applications involving uncertainty.
Our approximate feedback is computed \emph{a priori} and is therefore independent of any specific realization of the uncertain dynamics. We have shown that our approximation coincides with the unique optimal control along the optimal trajectory. Moreover, the numerical experiments suggest that our method yields a control that is significantly more robust to variations in initial conditions than the optimal open-loop control.

Future research directions include the investigation of computational strategies for large-scale problems with high-dimensional uncertainty, the extension to optimal control problems over an infinite time horizon, and a sequential update rule of the feedback driven by observational data.

\bibliographystyle{plain}
\bibliography{references}

\end{document}